\documentclass[reqno,11pt]{amsart}
\usepackage[marginparwidth=75pt]
{geometry}
\usepackage{pgfplots}
\usepackage{mathtools,amssymb,amsthm,mathrsfs,color,lineno,paralist,graphicx,float}
\usepackage[colorlinks,
linkcolor=blue,
anchorcolor=green,
citecolor=blue,
]{hyperref}

\usepackage{tikz}
\usepackage{graphicx}

\usepackage{enumitem}
\usepackage[T1]{fontenc}
\usepackage[utf8]{inputenc}
\usepackage{subfig} 

\usepackage[justification = centering, labelsep =period]{caption} 

\usepackage{calc}
\definecolor{bleu1}{RGB}{0,57,128}
\def\bleu1{\color{bleu1}}

\usepackage{etoolbox}
\patchcmd{\section}{\normalfont}{\normalfont \bleu1}{}{}
\patchcmd{\subsection}{\normalfont}{\normalfont \bleu1}{}{}
\patchcmd{\subsubsection}{\normalfont}{\normalfont \bleu1}{}{}

\newtheorem{proposition}{Proposition}[section]
\newtheorem{theorem}{Theorem}[section]
\newtheorem{definition}{Definition}[section] 
\newtheorem{lemma}{Lemma}[section]
\newtheorem{remark}{Remark}[section]
\newtheorem{corollary}{Corollary}[section]

\newtheorem{example}{Example}
\newtheorem{Claim}{Claim}
\newtheorem{Problem}{Problem}

\newcommand{\Z}{{\mathbb Z}}
\newcommand{\C}{{\mathbb C}}
\newcommand{\R}{{\mathbb R}}
\newcommand{\Q}{{\mathbb Q}}
\newcommand{\T}{{\mathbb T}}
\newcommand{\id}{\operatorname{id}}
\newcommand{\F}{{\mathbb F}}
\newcommand{\N}{{\mathbb N}}

\newcommand{\SL}{\mathrm{SL}}
\newcommand{\PSL}{\mathrm{{PSL}}}
\newcommand{\SO}{\mathrm{{SO}}}
\newcommand{\Sp}{\mathrm{{Sp}}}
\newcommand{\HSp}{\mathrm{{HSp}}}
\newcommand{\Her}{\mathrm{{Her}}}
\newcommand{\GL}{\mathrm{{GL}}}
\newcommand{\U}{\mathrm{U}}

\usepackage{tikz,tikz-3dplot}
\tdplotsetmaincoords{80}{45}
\tdplotsetrotatedcoords{-90}{180}{-90}
\tikzset{surface/.style={draw=blue!70!black, fill=blue!40!white, fill opacity=.6}}

\usepackage{pgfplots}
\usepgfplotslibrary{polar}
	\usepgfplotslibrary{polar}
\pgfplotsset{compat=1.17}
\makeatletter
\tikzset{reuse path/.code={\pgfsyssoftpath@setcurrentpath{#1}}}
\makeatother

\begin{document}

\title[On hyperbolicity among regular symplectic cocycles]{Hyperbolicity for one-frequency analytic quasi-periodic (Hermitian)-symplectic cocycles}
\author{Duxiao Wang}
\address{Chern Institute of Mathematics and LPMC, Nankai University, Tianjin 300071, China}
\email{2542721099@qq.com}

\author{Disheng Xu}
\address{School of Science, Great Bay University and Great bay institute for advanced study, 
Songshan Lake International Innovation Entrepreneurship Community A5, Dongguan 523000, China}
\email{xudisheng@gbu.edu.cn}

\author{Qi Zhou}
\address{
Chern Institute of Mathematics and LPMC, Nankai University, Tianjin 300071, China
}

 \email{qizhou@nankai.edu.cn}

\begin{abstract}
We demonstrate the existence of an open dense subset within the class of real analytic one-frequency quasi-periodic $\mathrm{\Sp}(4,\mathbb{R})$-cocycles, characterized by either the distinctness of all their Lyapunov exponents or the non-zero nature of all their accelerations, which partially answers an open problem raised by A. Avila.

		\end{abstract}
	\maketitle

\section{Introduction}

Understanding hyperbolic behavior is a central problem in the study of dynamical systems. Generally, hyperbolicity in dynamical systems is characterized through certain geometric or statistical properties of linear cocycles. In its simplest form, a linear (or matrix) cocycle consists of a measure-preserving dynamical system \( f : (X, \mu) \to (X, \mu) \) together with a matrix-valued function \( A : X \to G \), where \( A \) is typically assumed to be at least continuous, and \( G \) is a matrix Lie group (often referred to as the \textit{structure group} of the cocycle). The iteration of the cocycle has the form $A^n(x):=A(f^{n-1}(x))\cdots A(x)$. More generally, a linear cocycle can be seen as a morphism of vector bundles projecting to a transformation \( f \). 
\begin{example}[Derivative cocycles]
Let $f$ be a volume preserving and orientation preserving diffeomorphism on $\T^n$, then the tangent map $Df: \T^n \to \SL(n,\R)$ can be  identified as a linear cocycle over $f:(\T^n, \mathrm{Leb})\to (\T^n, \mathrm{Leb})$, which is called the \textit{derivative cocycle} associated to $f$. 

\end{example}
\begin{example}[Random products of matrices (random walks) on a matrix Lie group]
 Suppose $p$ is a probability measure on a matrix Lie group $G$, then the random products of i.i.d. matrices with distribution $p$ can be viewed as a linear cocycles $F$ over the shift map $\sigma$ on $(X,\mu),  X:=G^\N, \mu:=p^{\otimes \N}$ and  $F:X\to G, (X_0, X_1,\dots)\mapsto X_0$. 
\end{example}
\begin{example}[Quasiperiodic cocycles]Let $X=\T^n, \mu=\mathrm{Leb}$, a quasiperiodic $G$-cocycle is defined by a pair $(\alpha, A)$ where $\alpha\in \R^n$ be rationally independent and $A:\T^n\to G$. 
\end{example}

Linear cocycles have been extensively studied in various contexts, including differential dynamical systems 
\cite{ViaLec}, random walks on Lie groups and homogeneous dynamical systems\cite{BQ}, and spectral theory of quasiperiodic Schr\"odinger operators \cite{DF1,DF2,JICM}. 

The hyperbolicity of a linear cocycle is often described in terms of its Lyapunov exponents. Recall that the \( k \)-th Lyapunov exponent of a linear cocycle \((f, A, \mu)\) is defined as,
$$L_k(A) := \lim_{n\to \infty} \frac{1}{n}\int \ln \sigma_k(A^n(x))d\mu(x),$$
where $\sigma_k(A)$ is the $k$-th singular value of $A$. 
Several levels of hyperbolic behavior can be classified as follows:
\begin{enumerate}
    \item \textbf{Positive Lyapunov exponents:} At least one Lyapunov exponent is positive.
    \item \textbf{Non-uniform hyperbolicity:} There are no zero Lyapunov exponents.
    \item \textbf{Uniform hyperbolicity:} There exists a continuous invariant splitting \( E^s \oplus E^u \) such that the cocycle $A$ restricted to \( E^u \) grows exponentially fast, while \( A|_{E^s} \) contracts exponentially fast.
        \item \textbf{Simple Lyapunov spectrum:} All Lyapunov exponents are distinct.
\end{enumerate}

In the context of linear cocycles, a fundamental question concerns the density of hyperbolic behavior: given a fixed base dynamics \( f \), can we perturb the cocycle \( A \) within a specified regularity class to achieve a different (typically more hyperbolic) behavior?

This problem depends on three key ingredients:
\begin{enumerate}
\item \textbf{Structure group:} The properties of the structure group \( G \) critically influence the hyperbolic behavior of the cocycle. For instance, if \( G \) is compact, hyperbolicity cannot occur. If \( G = SO(m, n) \) or \( SU(m, n) \), the middle \( m - n \) Lyapunov exponents must be zero. Conversely, for  $G = \SL(2, \mathbb{R})$, Avila \cite{Av2011} demonstrated a general (assuming $f$ is not periodic on $\mathrm{supp}\mu$) density result for positive Lyapunov exponents in $ \SL(2, \mathbb{R})$ cocycles (See also \cite{Kri}, \cite{F-K}, etc.). This result was later generalized to real and Hermitian symplectic cocycles in \cite{DiSheng Xu}. However, for other groups, such as $G = \SL(2, \mathbb{C})$ or $\SL(3,\mathbb{R})$, extending these results remains very challenging. For example, there is a famous open question dates back to A. Avila\footnote{Private communication dates back to 2015.}, that whether positive Lyapunov exponents are dense for $ \SL(2, \mathbb{C})$ or $\SL(3,\mathbb{R})$ cocycles, without further assumptions.
    \item \textbf{Base dynamics:} When the base dynamics exhibit some hyperbolicity, cocycles in general position can often ``inherit'' hyperbolicity from the base. For strongly hyperbolic base dynamics, such as random matrix products, the simplicity of the Lyapunov spectrum has been investigated in works by Furstenberg \cite{Fur63}, Furstenberg and Kesten \cite{FK60}, Guivarc'h and Raugi \cite{GR} among others. In particular when the support of the distribution of random matrices is Zariski dense in $\SL(n,\mathbb R)$, Gol’dsheid-Margulis showed the associated Lyapunov spectrum is simple \cite{GM}. And A. Avila and M. Viana \cite{AV} gave a criterion of simplicity of Lyapunov spectrum for linear cocycles over Markov map and proved the Zorich-Kontsevich conjecture.
    
    For deterministic system with uniform or non-uniform hyperbolicity, see Bonatti-Viana \cite{BV} and Viana \cite{V} among others. For weaker hyperbolicity (e.g., partial hyperbolicity), similar results have been proved by Avila, Santamaria and Viana using techniques from partially hyperbolic systems \cite{ASV}. 
    
    In general, the weaker the hyperbolicity of the base, the weaker the hyperbolicity one can ``borrow'' from the base for the cocycle. Of course there are many dynamical systems without any hyperbolicity. A typical one is quasiperiodic systems. The story is essentially different, (before \cite{Av2011}), Avila proved that there is a dense set of real-analytic quasiperiodic $\SL(2,\mathbb{R})$-cocycle with a positive Lyapunov exponent. The key technique in the proof is also used in the solution of the Ten Martini theorem by Avila-Jitomirskaya \cite{AJ}.
    \item \textbf{Regularity assumptions:} The ability to perturb cocycles to achieve hyperbolicity depends on the regularity class considered. Lower regularity classes typically allow for easier perturbations, see Bochi-Viana theorem which completes a program proposed by Mane \cite{Bo}, \cite{BoVi} for an example for $C^0$ cocycles; while higher regularity, such as analyticity, poses significant challenges. 
    
\end{enumerate}

In summary, for general structure groups \( G \), non-hyperbolic base dynamics (e.g. systems with zero entropy, in particular rotations on tori), and high regularity assumptions (e.g., analyticity), the problem of density or prevalence of hyperbolic behavior remains extremely difficult, particularly for levels (2), (3) and (4), except in special cases like \( G = \SL(2, \mathbb{R}) \), where (1), (2), (4) are equivalent.


\subsection{Main results}
In this paper, we make progress on levels (2), (3), and (4) for real analytic one-frequency quasi-periodic \( G \)-cocycles, where \( G=\Sp(2d,\R) \)  or $\HSp(2d)$.  Recall that  $\Sp(2d,\R)$ (resp. $\HSp(2d)$) denotes  
the set of symplectic (resp. hermitian symplectic) matrices, i.e. 
    \begin{center}
        $\Sp(2d,\R)=\{A \in \R^{2d \times 2d}: A^{T}JA=J\},\qquad \HSp(2d)=\{A \in \C^{2d \times 2d}: A^{*}JA=J\}$.
    \end{center}    
where $J=\begin{pmatrix}
       O & -I_d \\
       I_d & O 
   \end{pmatrix}$ denotes the standard symplectic structure on $\R^{2d}$ (resp. $\C^{2d}$). 
By \cite{DiSheng Xu}, the quasiperiodic cocycles $ (\alpha,A) $ with $L_1(A)>0$ are dense in $C^\omega(\mathbb{T},G)$, after that, A. Avila raised us the following open problem:

\begin{Problem} Let $\alpha \in \R\backslash\Q$, $ G=\Sp(2d,\mathbb{R}) $ or $ G=\HSp(2d) $, whether $ (\alpha,A) $ with simple Lyapunov exponents are open and dense in $C^\omega(\mathbb{T},G)$? \end{Problem}

To understand this problem, the fundamental result belongs to Avila-Jitomirskaya-Sadel \cite{1D}, who proved that there exists an dense open subset $\mathcal{U}\subset C^\omega(\mathbb{T},\GL(m,\C))$, such that for every $A\in \mathcal{U}$, the Oseledets filtration of $ (\alpha,A) $ is either trivial (all Lyapunov exponents are equal) or dominated. Here an invariant continuous decomposition $\mathbb{C}^{m} = E^1(x) \oplus \cdots \oplus E^k(x)$ is called dominated if there exists $n\geq 1$ such that for any unit vectors $w_j\in E^j(x) $ we have $\|A_n(x)\cdot w_j\|> \|A_n(x)\cdot w_{j+1}\|$.

The key idea to get this kind of hyperbolicity comes from regularity property, which involves the holomorphic extension of the cocycle dynamics, and was first introduced (in the particular case of $\SL(2,\C)$-cocycles) by Avila \cite{Global}. Note for any $ B\in \GL(m,\C) $, the $ k $-th exterior product $ \Lambda^kB $ satisfies $ \sigma_1(\Lambda^kB)=\|\Lambda^kB\| $, leading to the expression $ L^k(A)=\sum\limits_{j=1}^kL_j(A) $, where $$
L^k(A)=\lim\limits_{n\rightarrow \infty}\frac{1}{n}\int_{\T}\ln\|\Lambda^kA_n(x)\|dx.
$$ By analyticity, one can extend $A(x)$ to a strip $|Im x|<\delta$ in the complex plane. We say that a cocycle $ (\alpha,A) $ is $ k $-\textbf{regular} if and only if $ y \mapsto L^k(A(\cdot+iy)) $ is an affine function of $ y $ near $ 0 $, and we say $(\alpha,A)$ is \textbf{irregular}, if for any $1\leq k\leq m$, $ (\alpha,A) $ is not $ k $-regular. Avila-Jitomirskaya-Sadel \cite{1D} proved that for every $t \neq 0$ small enough, $(\alpha,A(\cdot+it))$ is $k$-regular, which actually gives the hyperbolicity (dominated splitting), by assuming $L_k(A)>L_{k+1}(A)$. One should note the former result means that the convex functions $ y \mapsto L^k(A(\cdot+iy)) $ are in fact piecewise affine, which is connected to a quantization phenomenon of \textit{acceleration}, the key concept of Avila's global theory \cite{Global,1D}: $$
\omega^k(A)=\lim\limits_{y\rightarrow 0^+}\frac{1}{2\pi y}(L^k(A(\cdot+iy))-L^k(A)), \qquad \omega_k(A)= \omega^k(\alpha,A)-\omega^{k-1}(\alpha,A).
$$

Let's back to Avila's problem, and first concentrate on the symplectic case. Compared to \cite{1D}, there are two fundamental issues which need to be solved. First, how to avoid holomorphic extension (othewise the cocycle is not symplectic anymore) to get hyperbolicity. Second, how to perturbate the cocycle to have distinct Lyapunov exponents. Our basic observation here is that owing to the symplecity, without holomorphic extension, 
 one can get hyperbolicity (even simple Lyapunov exponent) from  regularity itself. 
 Note for $ A \in C^\omega(\T, \Sp(2d,\R)) $, $ (\alpha,A) $ is $ k $-regular if and only if $ \omega^k(A)=0 $.
Now, we partially answer A.Avila's open problem as follows: 
\begin{theorem}\label{SIMPLE DENSITY}
Let $\alpha\in \R\backslash\Q$. The  cocycles $ (\alpha,A) $ with simple Lyapunov exponents are open and dense in the set 
of \( 1 \)-regular or \( 2 \)-regular  $\Sp(4,\mathbb{R})$-cocycles, i.e., in the set
\begin{center}
            $\{A \in C^\omega(\mathbb{T},\Sp(4,\mathbb{R})):\omega^1(A)\omega^2(A)=0 \}.$
        \end{center}
\end{theorem}

As a direct consequence, we have the following:
\begin{corollary}
    Let $\alpha\in \R\backslash\Q$. There exists a dense subset $\mathcal{V}$ of $A \in C^\omega(\mathbb{T},\Sp(4,\mathbb{R}))$, such that if $A\in \mathcal{V}$, either $(\alpha,A)$ has simple Lyapunov exponents or $(\alpha,A)$ is irregular. 
\end{corollary}

Once we have this, in the $\Sp(4,\mathbb{R})$-case, Problem 1 reduces to the following:

\begin{Problem} 
    Let $\alpha\in \R\backslash\Q$, $A \in C^\omega(\mathbb{T},\Sp(4,\mathbb{R}))$.
Whether $ (\alpha,A) $ with simple Lyapunov exponents is  dense in the set of irregular cocycles?
\end{Problem}

\begin{remark}
 One can understand this problem in the following way, even back to $\SL(2,\R)$, it is still open whether uniformly hyperbolic cocycles are dense in supercritical cocycles ($L(A)>0, \omega(A)>0$), i.e. in the positive Lyapunov exponents case, whether one can perturbate  
the cocycle with $\omega(A)>0$ to $\omega(A)=0$.    
\end{remark}

As another direct consequence, the cocycles $(\alpha,A)$ with simple Lyapunov exponents are open and dense in $ \mathcal{R}(\Sp(4,\mathbb{R})) $, where $ \mathcal{R}(G) $ denotes the set of all $k-$regular cocycles, defined as $$
\mathcal{R}(G)=\{A \in C^\omega(\T,G): (\alpha, A) \text{ is } k\text{-regular for all } k \in \{1,2,\dots,d\}\}.
$$
Despite our inability to prove the density of cocycles with simple Lyapunov exponents in $ \mathcal{R}(\HSp(2d)), d\geq 2 $ or $ \mathcal{R}(\Sp(2d,\R)), d\geq 3 $, we are able to establish the following weaker result:

\begin{theorem}\label{UH DENSITY}
Let $\alpha\in \R\backslash\Q$,  \( G=\Sp(2d,\mathbb{R}) \) or \( G=\HSp(2d) \), then uniformly hyperbolic cocycles are open and dense in \( \mathcal{R}(G) \).
\end{theorem}

To conclude, this work represents a useful attempt forward in the study of real analytic one-frequency quasi-periodic $ G $-cocycles, particularly in understanding the density of hyperbolicity cocycles.

\subsection{Analytic diagonalization of Hermitian matrix and Analytic Sylvester Inertia Theorem}\label{adp}

Another noteworthy result, which emerged during our work and may hold independent interest, pertains to the analytic Sylvester Inertia Theorem. Before explaining this result, let's recall another related result: analytic diagonalization of Hermitian matrix-valued functions. Assume $ G(\cdot) \in C^\omega(I,\Her(m,\C)) $, where $ I $ is an open interval. Analytic diagonalization seeks to determine the existence of a matrix $ U \in C^\omega(I,U(m)) $ such that $ U(x)^*G(x)U(x) $ is diagonal. In cases where the eigenvalues of $ G(\cdot) $ do not cross (i.e., they remain distinct in a neighborhood of each point), the use of the implicit function theorem enables a straightforward analytic diagonalization. Consequently, the eigenvectors corresponding to distinct eigenvalues can likewise be chosen to be analytic. When there are repeated eigenvalues, the situation is more complicated, 
however Kato \cite{Kato} still showed that $G(\cdot)$ can be diagonalized analytically.

This result was later generalized for cases where $ G $ takes on values in normal matrices, following an observation by Bluster \cite{Kato}. Further extensions of this result, concerning regularity, have been made to encompass cases where $ G(\cdot) $ is in the $ C^\infty $ topology \cite{AKLM} and where $ G $ is in an ultra-differentiable topology \cite{KMR}. Additional results can be found in \cite{KL, Rai1, Rai2}. Such results are fundamental to spectral theory and perturbation analysis, with many important applications \cite{APP1,APP2,APP3}.

Nevertheless, the above generalizations do not address periodicity. A pertinent question arises: if $ G(\cdot) $ is periodic (i.e., $ G \in C^\omega(\mathbb{T}, \Her(m,\C)) $), can we still find a matrix $ U \in C^\omega(\mathbb{T}, \U(m)) $ such that $ U(x)^*G(x)U(x) $ is diagonal? In general, this is not the case due to the existence of degenerate eigenvalues. For example, consider the matrix
$$
G(x)=\begin{pmatrix}
e^{i \pi x}& -e^{i\pi x} \\
1 & 1 \\
\end{pmatrix}\begin{pmatrix}
2+\sin(\pi x) & 0 \\
0 & 2-\sin(\pi x) \\
\end{pmatrix}\begin{pmatrix}
e^{-i \pi x}& 1 \\
-e^{-i\pi x} & 1 \\
\end{pmatrix}.
$$
It can be verified that $ G \in C^\omega(\mathbb{T}, \Her(2,\C) \cap \GL(2,\C)) $, but there is no $ U \in C^\omega(\mathbb{T}, \U(2)) $ such that $ U(x)^*G(x)U(x) $ is diagonal since $ 1 $ is not a period of its eigenvalues. Nevertheless, we can establish the following result:

\begin{proposition}
\label{diagonalization}
Let \( G \in C^\omega(\mathbb{T}, \Her(m,\C)) \). Then there exists $K \in \mathbb{N}_+, U \in C^\omega(K\mathbb{T},$  $\U(m)) $ such that \( U(x)^*G(x)U(x) \) is diagonal.
\end{proposition}

This indicates that period doubling occurs in the analytic diagonalization of periodic Hermitian matrices. However, if we allow
 congruence transformation (then $U$ is not unitary anymore), 
then  $ U(\cdot) $ share the same periodicity as $ G(\cdot) $:

\begin{theorem}
    
\label{diagonalization-1}
Let \( G \in C^\omega(\mathbb{T}, \Her(m,\C) \cap \GL(m,\C)) \). Then there exists $ N \in C^\omega(\mathbb{T},$ $\GL(m,\C)), p,q \in \mathbb{N}_+ $ such that \( N(x)^*G(x)N(x)=\operatorname{diag}(I_p,-I_q). \)
\end{theorem}

Let's recall Sylvester's law of inertia, which states that the signature (the number of positive, negative, and zero eigenvalues) of a Hermitian matrix is invariant under congruence transformations. The proposition extends this to an analytic setting, demonstrating that an analytic family of invertible Hermitian matrices can be analytically congruent to a constant diagonal matrix with fixed signature 
$\operatorname{diag}(I_p,-I_q)$. That's how this subsection got the name. Note
 this proposition will play a critical role in the discussion of analytic block-diagonalization of Hermitian-symplectic cocycles.

\subsection{Difficulty, ideas of the proof}

One of our primary motivations is to analyze the dynamical behavior of symplectic (resp. Hermitian-symplectic) analytic quasiperiodic cocycle $(\alpha,A)$ when it is partially hyperbolic, i.e., $\mathbb{R}^{2d} = E^s(x) \oplus E^c(x) \oplus E^u(x)$ (resp. $\mathbb{C}^{2d} = E^s(x) \oplus E^c(x) \oplus E^u(x)$). Symplectic (resp. Hermitian-symplectic) quasiperiodic cocycle $(\alpha,A)$ appear quite naturally in the study of spectral theory of quasi-periodic Schr\"odinger operators, indeed, taking Aubry duality of Schr\"odinger operator induces a special symplectic (resp. Hermitian-symplectic) cocycle \cite{HP}, and symplecticity or Hermitian-symplecticity depends on the potential is even or not. We should emphasize partially hyperbolic cocycles plays the essential role in the study of quantitative version of Avila's global theory \cite{GJYZ}, and its spectral applications \cite{Ge-Jito-You,Ge-Jito}. Our whole proof include two parts: block-diagonalize the partially hyperbolic cocycles, get distinctness of the Lyapunov exponents.

\subsubsection{Analytic block-diagonalization of partially hyperbolic cocycle.}
In the symplectic context, Herman \cite{H} first observed that the continuous splitting of $\mathbb{R}^2$ associated with a real uniformly hyperbolic cocycle can be topologically non-trivial. Specifically, if $E(x) = \operatorname{span}_{\mathbb{R}}(u(x))$ denotes the stable direction, it is possible that $u(x+1) = -u(x)$. To describe this phenomenon, we introduce the following definition:
\begin{definition}
A function $u \in C^\omega(\mathbb{R},\mathbb{R}^d)$ is called an \textbf{antiperiodic} analytic real vector if  $u(x+1) = -u(x)$ for every $x \in \mathbb{R}$.
\end{definition}
In the symplectic case, our primary observation is that any real periodic analytic subspace possesses a real analytic basis in which all vectors are periodic, with the possible exception of one antiperiodic vector (as stated in Proposition \ref{REAL LIFT}). It is also worthwhile to compare Proposition \ref{DIAGNALIZATION LEMMA} with  Ge-Jitomirskaya-You \cite{Ge-Jito-You}, who only trivialize the central bundle. However, when aiming to block-diagonalize the cocycle, it is important to note that the phenomenon of antiperiodicity is unavoidable.

However, in the Hermitian - symplectic context, the Hermitian - symplectic analogue of Proposition \ref{REAL SYMPLECTIC LIFT} presents significantly more challenges. Hermitian - symplectic subspaces of $\C^{2d}$ are more complex than their symplectic counterparts in $\R^{2d}$.
One primary reason for this complexity is that, as shown by Harmer \cite{Harmer}, there exist Hermitian-symplectic subspaces for which it is impossible to identify a set of canonical basis (A set $\{v_1,\cdots,v_{n},v_{-1},\cdots,v_{-n}\} \subset \C^{2d}$ is called a canonical basis of the subspace spanned by it if it satisfies $v_{-j}^*Jv_{k}=\delta_{jk}$ for all $k$ and $j \geq 0$). In fact, the Krein matrices of Hermitian-symplectic subspaces with a canonical basis have a signature of zero (for more in-depth discussions, refer to Section \ref{abd}). This is precisely why the Analytic Sylvester Inertia Theorem (Section \ref{adp}) comes into play.

\subsubsection{Distinctness of the Lyapunov exponents.}To get simple Lyapunov spectrum, our proof hinges on several recent key advances of  Avila's gloabl theory of  one-frequency analytic $\SL(2,\R)$ cocycles \cite{Global,Aac,avila2,ARC}. To elaborate, once we have the analytical block-diagonalization result (Proposition \ref{DIAGNALIZATION LEMMA}), Theorem \ref{SIMPLE DENSITY} simplifies to whether, for a uniformly hyperbolic cocycle $(\alpha,A)$, i.e., there exist $T \in C^\omega(2\T,\Sp(4,\R))$, $\lambda \in C^\omega(\T,\R)$, $\Lambda \in \mathfrak{A}$ such that $$T(x+\alpha)^{-1}A(x)T(x) = \begin{pmatrix}
\lambda(x)\Lambda(x) & O \\
O & \lambda(x)^{-1}\Lambda(x)^{-\top}
\end{pmatrix},$$ we can perturb the cocycle $(\alpha,\Lambda) \in \R\backslash\Q \times \mathfrak{A}$ to have positive Lyapunov exponents. Here, $\mathfrak{A}$ is defined as $$\mathfrak{A} = \{\Lambda \in C^\omega(2\T,\SL(2,\R)): \Lambda(x+1) = \begin{pmatrix}
-1 & 0 \\
0 & 1
\end{pmatrix}\Lambda(x)\begin{pmatrix}
-1 & 0 \\
0 & 1
\end{pmatrix}\},$$ which is a submanifold of $C^\omega(2\T,\SL(2,\R))$.
As one will find, any $\Lambda \in  \mathfrak{A}$ have zero topological degree,  indicating that the traditional method \cite{JMD,DiSheng Xu}, technique based on monotonic perturbation \cite{AK} or Kotani's theory \cite{Ko, KoSi}, is not applicable here.

However, according to Avila's global theory \cite{Global},  zero Lyapunov cocycles are classified into two classes: critical cocycles ($L(A)=0, \omega(A)>0$) and subcritical ones ($L(A)=0, \omega(A)=0$). Moreover, as was shown by Avila \cite{Global}, critical cocycles are confined to a countable union of codimension-one analytic submanifolds within $C^{\omega}(\T,\SL(2,\R))$. Thus, if $(\alpha,\Lambda) \in \R\backslash\Q \times \mathfrak{A}$ is critical, the challenge shifts to determining whether critical cocycles are also rare in $\mathfrak{A}$, noting that $\mathfrak{A}$ is not open in $C^{\omega}(\T,\SL(2,\R))$.

If the dynamical system $(\alpha,\Lambda)$ is subcritical, the Almost Reducibility Theorem (ART) by Avila becomes a crucial tool. ART is a global result that asserts that a cocycle $(\alpha,\Lambda)$ is almost reducible if it is subcritical \cite{Global}. Recall that $(\alpha,A_1)$ is conjugated to $(\alpha,A_2)$, if there exists $B\in C^\omega(\T,\mathrm{PSL}(2,\R))$
such that 
$$B^{-1}(x+\alpha)A_1(x)B(x)=A_2(x).$$ 
Then $(\alpha,A)$ is almost reducible if the closure of its analytic conjugate class contains the constant. 
This theorem was initially announced in \cite{Global} and subsequently proven in \cite{Aac,avila2}, consult also another proof for the Sch\"odinger case \cite{ge}.
Lemma \ref{zrho} establishes that the fibered rotation number of the cocycle is rational with respect to the frequency $\alpha$. ART then essentially implies that $(\alpha,\Lambda)$ is situated at the boundary of the set of uniformly hyperbolic cocycles\footnote {In the context of Schr\"odinger operators, this translates to $(\alpha,\Lambda)$ being at the boundary of the spectral gaps.}. Thus, the primary challenge shifts to understanding how to perturbate the cocycles while preserving their group structure of $\mathfrak{A}$.
The situation becomes more complicated when the frequency is Liouvillean. In such cases, $(\alpha,\Lambda)$ is only almost reducible. It's important to note that if the frequency $\alpha$ is Diophantine, then $(\alpha,\Lambda)$ is indeed reducible. In these scenarios, quantitative estimates of almost reducibility become essential. These estimates are pivotal in addressing the non-critical aspect of the Dry Ten Martini Problem, as discussed in \cite{ARC}.

\section{Preliminaries}

\subsection{Continued fraction expansion} Let $\alpha\in \R \backslash\Q$, $a_0:=0$ and $\alpha_0:=\alpha$. Inductively, for $k\geq 1$, we define
$$
a_k:=[\alpha_{k-1}^{-1}], \  \ \alpha_k=\alpha_{k-1}^{-1}-a_k.
$$
Let $p_0:=0$, $p_1:=1$, $q_0:=1$, $q_1:=a_1$. Again inductively, set
$p_k:=a_kp_{k-1}+p_{k-2}$, $q_k:=a_kq_{k-1}+q_{k-2}$. Then $q_n$ are
the  denominators of the best rational approximamts of $\alpha,$ since
we have $\|k\alpha\|_{\R/\Z}\geq \|q_{n-1}\alpha\|_{\R/\Z}$ for all
$k$ satisfying $1\leq k< q_n$. We also have 
$$
\frac{1}{2q_{n+1}}\leq \|q_n\alpha\|_{\R/\Z}\leq \frac{1}{q_{n+1}}.
$$

\subsection{ Fibered rotation number}

Assume now $ A\in C^\omega(\T,\mathrm{\SL}(2,\R)) $ is homotopic to the identity, 
then the continuous
map
\begin{eqnarray*}    
F : \T \times \mathbb{S}^1  &\rightarrow& \T \times \mathbb{S}^1,\\
(x,v) &\mapsto&
(x + \alpha,
\frac{A(x)v}{\|A(x)v\|})
\end{eqnarray*}
is also homotopic to the identity, thus it admits a continuous lift
$\widetilde{F} : \T \times \R \rightarrow \T \times \R $
of the form
$F(x,t) = (x + \alpha,t + f(x,t))$ with
$f \in C(\T \times \R, \R)$ and $
f(x,t + 1) = f(x,t)$.
Then the \textit{fibered rotation number} of $ (\alpha,A) $ is defined by
$$\rho(A):=\lim_{N \rightarrow \infty}\frac{1}{N}\sum_{n=0}^{N-1}f(\widetilde{F}^n(x,t)) \mod \Z$$
which does not depend on the choices of  $(x,t) \in \T \times \R$ and the lift $\widetilde{F}$ \cite{johonson and moser,H}.

\subsection{Uniform hyperbolicity and dominated splitting}\label{3.3}

Let $\F=\R$ or $\C$,  and we say that  $(\alpha,A)$ is $k$-dominated for some $1 \leq k \leq d-1$, if there exists a dominated decomposition $\F^d=E^+(x) \oplus E^-(x) $ with $\dim E^+=k$. If $\alpha\in \R \backslash\Q$, then it follows from the definitions that the Oseledets splitting is dominated if and only 
if $(\alpha,A)$ is $k$-dominated for each $k$ such that $L_k(\alpha, A)>L_{k+1}(\alpha, A)$.
The next  result gives the basic relation between regularity and domination:
\begin{theorem}\cite{1D}\label{DOMINATE=REGULAR}
    If $1 \leq k \leq d-1$ is such that $L_k(\alpha, A)>L_{k+1}(\alpha, A)$ then $(\alpha, A)$ is $k$-regular if and only if $(\alpha, A)$ is $k$-dominated.
\end{theorem}

We also recall $(\alpha,A)$  is uniformly hyperbolic, if there exists a continuous   decomposition $\F^d=E^+(x) \oplus E^-(x) $, such that 
$$
\begin{aligned}
\| A_n(x)v\| \leqslant Ce^{-cn}\| v\|, \quad & v\in E^-(x),\\
\| A_n(x)^{-1}v\| \leqslant Ce^{-cn}\| v\|,  \quad & v\in E^+(x+n\alpha).
\end{aligned}
$$ 
for some constants $C>0,c>0$. This splitting is  invariant by the dynamics, i.e. 
$
A(x)E^{\pm}(x)=E^{\pm}(x+\alpha).
$

 If $G$ is a subgroup of $\F^{m \times m}$, let $\mathcal{D O} _k(G)$ denote the set of $k-$dominated $C^\omega(\T,G)$-cocycles. Thus for the case $G=\Sp(2d,\R)$ or $G=\HSp(2d)$, the  cocycle $(\alpha, A)$ is {\it uniformly hyperbolic} if and only if it is $d$-dominated. The set of uniformly hyperbolic cocycles is denoted by $\mathcal{UH}(G)$, which is an open set of $C^\omega(\T,G)$ in the $C^\omega$-topology. Moreover,  the dominated decomposition can be proven depending analytically on $A \in \mathcal{D O} _k(G)$. Denote $Gr(n,\C^d)$ the the Grassmannian
of $n$-dimensional subspaces of $\C^d$, then we have the following:
 
\begin{theorem}\cite{1D}\label{HOLOMORPHIC SECTION}
Let $G$ is a subgroup of $\C^{d\times d}$. Denote the unstable and stable subspace of $(\alpha,A)$ by $E^+_A(x)$ and $E^-_A(x)$ correspondingly.
\begin{itemize}
    \item For any $x \in \T$, two maps $A \mapsto E^+_A(x)$ and $A \mapsto E^-_A(x)$ is a holomorphic function of $A \in \mathcal{D O}_k(G)$.
    \item Moreover, let $(\alpha, A(\cdot+i t))$ is $k$-dominated for $t \in\left(t_{-}, t_{+}\right)$, then $z \mapsto E^{+}_{A}(z) \in C^\omega(\T,Gr(k,\C^{d}))$, and  $z \mapsto E^{-}_{A}(z) \in C^\omega(\T,Gr(d-k,\C^{d}))$ are holomorphic for $z=x+it$, $t \in\left(t_{-}, t_{+}\right)$.
\end{itemize}
\end{theorem}



\section{Analytic Block diagonalization}

The goal of this section is to analytically block diagonalize one frequency quasiperiodic partially hyperbolic cocycles. We start with the following  quite simple observation:
\begin{lemma}
    Let $\alpha\in \R \backslash\Q$, $G=\Sp(2d,\R)$ or $G=\HSp(2d)$, then $\mathcal{DO}_n(G)=\mathcal{DO}_{2d-n}(G)$.
\end{lemma}
\begin{proof}
    If $G=\Sp(2d,\R)$, by symmetry, we have $L_n(A(\cdot+iy))=-L_{2d-n}(A(\cdot+iy))$. If $G=\HSp(2d)$, as $A(x-iy)^*JA(x+iy)$  is a holomorphic function with respect to $x+iy\in \T_\delta$ which equals to $J$ whenever $y=0$, thus for any $y \neq 0$ it holds that $A(x-iy)^*JA(x+iy)=J$, it follows that $L_n(A(\cdot+iy))=-L_{2d-n}(A(\cdot-iy))$.  By Theorem \ref{DOMINATE=REGULAR}, the result follows. 
\end{proof}

Thus, if $(\alpha,A)$ is $n$-dominated,  
then we actually get a dominated splitting 
$E^u\oplus E^c\oplus E^s, \dim E^s=\dim E^u=n$ 
such that $E^s\oplus E^c$ is the stable bundle and $E^u$ is the unstable bundle with respect to $n$-dominated splitting, while $E^u\oplus E^c$ is the unstable bundle and $E^s$ is the stable distribution with respect to $(2d-n)-$dominated splitting. Moreover, we have the following:
\begin{lemma}\label{U,S,C symplectic properties}
Let $\alpha\in \R\backslash \Q$, $A\in C^{\omega}(\T, G)$, where \( G=\Sp(2d,\mathbb{R}) \) or \( G=\HSp(2d) \).  
If $(\alpha, A)\in\mathcal{DO}_n(G)$, then we have the following:
\begin{enumerate}

    \item $E^u(x)\oplus E^s(x)$ and $E^c(x)$ are symplectic-orthogonal (or hermitian-symplectic-orthogonal) to each other.
     \item $E^u(x)\oplus E^s(x)$ and $E^c(x)$ are symplectic subspaces (or hermitian-symplectic subspaces). 
      \item $E^u(x)$ and $E^s(x)$ are isotropic subspaces (or hermitian-isotropic subspaces).
\end{enumerate}
\end{lemma}
\begin{proof}
To simplify our analysis, we focus on the case $ G = \HSp(2d) $. Let $(\alpha, A) \in \mathcal{DO}_n(G)$ be $n$-dominated, and assume $L_{n+1}(A) < l_1 < l_2 < L_n(A)$. By definition, consider $v^c_x \in E^c(x)$, $v^u_x \in E^u(x)$, and $v^s_x \in E^s(x)$. There exist constants $C_c, C_s > 0$ such that:
$$
\|A_k(x)v^c_x\| < C_c e^{kl_1} \quad \text{and} \quad \|A_k(x)v^s_x\| < C_s e^{-kl_2} \text{ for } k \in \mathbb{N}_{+}.
$$
Since $A_k(x) \in \HSp(2d)$, we have the relation:
$$
{v_x^c}^*Jv_x^s = {v_x^s}^*A_k(x)^*JA_k(x)v_x^s.
$$
Consequently, as $k \rightarrow +\infty$, we observe that:
$$
\begin{aligned}
\|{v_x^c}^*Jv_x^s\| &= \|{v_x^c}^*A_k(x)^*JA_k(x)v_x^s\| 
\leq \|(A_k(x))v_x^c\| \|(A_k(x))v_x^s\| \leq C_c C_s e^{-k(l_1-l_2)} \rightarrow 0.
\end{aligned}
$$
Therefore, $\|{v_x^c}^*Jv_x^s\| = 0$ implies that $E^c$ is (hermitian-)symplectic-orthogonal to $E^s$.

Using a similar argument, by letting $k \rightarrow -\infty$, we find that:
$
{v_x^c}^*Jv_x^u = 0,
$
which shows that $E^c$ is (hermitian-)symplectic-orthogonal to $E^u$. Consequently, $E^c$ and $E^u \oplus E^s$ are also (hermitian-)symplectic-orthogonal. Given that
$\operatorname{dim}E^c(x) + \operatorname{dim}(E^u(x) \oplus E^s(x)) = 2d$,
it follows that $E^c$ and $E^u \oplus E^s$ are symplectic complements of each other, thus establishing that both are (hermitian-)symplectic subspaces.

By the same reasoning, we have:
$$
\|{v_x^s}^*Jv_x^s\| \leq \|(A_k(x))v_x^s\| \|(A_k(x))v_x^s\| \rightarrow 0 \quad (k \rightarrow +\infty),
$$
and
$$
\|{v_x^u}^*Jv_x^u\| \leq \|(A_k(x))v_x^u\| \|(A_k(x))v_x^u\| \rightarrow 0 \quad (k \rightarrow -\infty).
$$
Thus, we conclude that $E^s(x)$ and $E^u(x)$ are hermitian-isotropic subspaces of $\mathbb{C}^{2d}$.

\end{proof}

In the context of the real symplectic case, as discussed in the introduction, the bundles $ E^s $ and $ E^u $ can exhibit topological non-triviality. It is intuitive to show that for any $ E \in C^\omega(\mathbb{T}, Gr(1, \mathbb{R}^d)) $, there exists a vector $ u \in C^\omega(\mathbb{R}, \mathbb{R}^d) $ that is either periodic or antiperiodic, such that $ E(x) = \operatorname{span}_{\mathbb{R}}(u(x)) $. Consequently, we can identify $ E(x) $ with $ u(x)$, that is,  we define the set
$$
\{u \in C^\omega(2\mathbb{T}, \mathbb{R}^d): u(x+1) = -u(x)\} \cup C^\omega(\mathbb{T}, \mathbb{R}^d)
$$
as $ C^\omega(\mathbb{T}, \mathbb{PR}^d) $.
Furthermore, we introduce a function
$$
\tau: C^\omega(\mathbb{T}, \mathbb{PR}^d) \rightarrow \{1, -1\}
$$
defined by the equation
$
u(1) = \tau(u) u(0).
$
For any vector $ u \in C^\omega(\mathbb{T}, \mathbb{PR}^d) $, the vector $ u $ is 1-periodic if and only if $ \tau(u) = 1 $, and antiperiodic if and only if $ \tau(u) = -1 $.
Additionally, for $ 1 \leq n < d $, the Grassmannian $ Gr(n, \mathbb{R}^d) $ can be characterized as
$$
Gr(n, \mathbb{R}^d) = \{\operatorname{span}_{\mathbb{R}}(v_1, \ldots, v_n): v_1, \ldots, v_n \in \mathbb{R}^d\},
$$
which can also be represented as the set of wedge products:
$$
\{\operatorname{span}_{\mathbb{R}}(v_1 \wedge \cdots \wedge v_n): v_1, \ldots, v_n \in \mathbb{R}^d\}.
$$
This representation forms a closed submanifold of $ \mathbb{PR}^{C_d^n} $. Therefore, we can naturally regard $ E \in C^\omega(\mathbb{T}, Gr(n,\mathbb{R}^d)) $ as an element of $ C^\omega(\mathbb{T}, \mathbb{PR}^{C_d^n}) $, which in turn ensures that $ \tau(E) $ is also well-defined.

To state precisely the result, we recall the direct sums of (hermitian)-symplectic groups. Let 
 $$S_1= \begin{pmatrix}
    A_1 & A_2 \\
    A_3 & A_4 \\
\end{pmatrix} \in \Sp(2n_1,\R) \qquad \text{and} \qquad S_2= \begin{pmatrix}
    B_1 & B_2 \\
    B_3 & B_4 \\
\end{pmatrix}\in \Sp(2n_2,\R),$$ then
      $$S_1 \diamond S_2=\begin{pmatrix}
    A_1 & O & A_2 & O \\
     O & B_1  & O & B_2 \\
     A_3 & O & A_4 & O \\
    O & B_3 & O & B_4 \\
\end{pmatrix}\in \Sp(2n_1+2 n_2,\R).$$ The mapping $(S_1,S_2)\rightarrow S_1 \diamond S_2 $ thus defined a group monomorphism $\Sp(2n_1,\R) \oplus \Sp(2n_2,\R) \rightarrow \Sp(2n_1+2 n_2,\R)$. Note while we state the action for symplectic groups, it also works for  (hermitian)-symplectic groups. Once we have this, we state the following analytic block diagonalization result: 
 \begin{proposition}\label{DIAGNALIZATION LEMMA}
Let $\alpha\in \R\backslash \Q$, $A\in C^{\omega}(\T, \Sp(2d,\R))$.  
Suppose  $A\in \mathcal{DO}_n(\Sp(2d,\R))$ is $n-$dominated, $ 1\leq n\leq d$, then there exists $\chi\in\{1,2\}$, $T \in C^\omega(\chi\T,\Sp(2d,\R))$, $\lambda \in C^\omega(\T,\R_+)$,
    $\Lambda \in C^\omega(\chi\T,\SL(n,\R))$ and $\Gamma \in C^\omega(\T,\Sp(2d-2n,\R))$ such that
    
    \begin{equation}\label{T(X+ALPHA)-1}        T(x+\alpha)^{-1}A(x)T(x)=\begin{pmatrix}
    \lambda(x)\Lambda(x) & O \\
    O & \lambda(x)^{-1}\Lambda(x)^{-\top}
\end{pmatrix}\diamond \Gamma(x).
    \end{equation}
with the following property: 
 \begin{enumerate}
    \item if  $\tau(E^u)=\tau(E^s)=1$, then $\chi=1$.
    \item if $\tau(E^u)=\tau(E^s)=-1$,  then $\chi=2$. Moreover, if we denote by $$
 \mathcal{T}_n=\begin{pmatrix}
    -1 &     \\
      &I_{n-1}
\end{pmatrix}, \mathcal{P}_d=\operatorname{diag}(\mathcal{T}_d,\mathcal{T}_d)=(-I_2)\diamond(I_{2d-2}),$$ then

    $$T(x+1)=T(x)\mathcal{P}_d, \Lambda(x+1)=\mathcal{T}_n^{-1}\Lambda(x)\mathcal{T}_n.$$

    \end{enumerate}
    
Especially if $n=d$, i.e.  $(\alpha, A)$ is uniformly hyperbolic, then 
\begin{equation}\label{T(X+ALPHA)-1.5}
    T(x+\alpha)^{-1}A(x)T(x)=\begin{pmatrix}
    \lambda(x)\Lambda(x) & O \\
    O & \lambda(x)^{-1}\Lambda(x)^{-\top}
\end{pmatrix}.
\end{equation} 
 \end{proposition}

While in the hermitian symplectic case, we have the following:
\begin{proposition}\label{DIAGNALIZATION LEMMA-1}
Let $\alpha\in \R\backslash \Q$, $A\in \mathcal{DO}_n(\HSp(2d))$. 
Then there exist $T \in C^\omega(\T,\HSp(2d))$, $\Lambda \in C^\omega(\T,\GL(n,\C))$ and $\Gamma \in C^\omega(\T,\HSp(2d-2n))$ such that 
\begin{equation}\label{T(x+alpha)-2}   T(x+\alpha)^{-1}A(x)T(x)=\begin{pmatrix}
    \Lambda(x) & O \\
    O & \Lambda(x)^{-*}
\end{pmatrix}\diamond \Gamma(x).
\end{equation}
\end{proposition}

\subsection{Proof of Theorem \ref{UH DENSITY}}
First we need the result perturbing the cocycle to one whose first Lyapunov exponent is positive. 
\begin{theorem}\cite{JMD}\label{JMD POSITIVE DENSITY}
 Let $\alpha\in \R\backslash \Q$. Then there is a dense set of $A \in C^\omega\left(\T, \mathrm{SL}(2, \mathbb{R})\right)$ (in the usual inductive limit topology) such that $L(\alpha, A)>0$.
\end{theorem}
\begin{theorem}\cite{DiSheng Xu}\label{POSITIVE DENSITY}
     Let $\alpha\in \R\backslash \Q$,  $G=\Sp(2d,\R)$ or $G=\HSp(2d)$ there is a dense set of $A \in C^\omega\left(\T,G \right)$ (in the usual inductive limit topology) such that $L_1(\alpha, A)>0$.
\end{theorem}

Once we have these, we prove Theorem \ref{UH DENSITY} by induction. 
    When $d=1$, it follows by Theorem \ref{JMD POSITIVE DENSITY} and Theorem \ref{DOMINATE=REGULAR}. If $d\geq 2$, 
by Theorem \ref{POSITIVE DENSITY}, we can take $\tilde{A}$ is  close to $A$ that $\tilde{A} \in \mathcal{R}(G)$ and $L_1(\tilde{A})>0$. If $L_d(\tilde{A})>0$ then by Theorem \ref{DOMINATE=REGULAR}, $\tilde{A} \in \mathcal{UH}(G)$. Otherwise, there exists $1 \leq n \leq d-1$,  such that $L_{n}(\tilde{A})>L_{n+1}(\tilde{A})=0$. Still by Theorem \ref{DOMINATE=REGULAR}, $\tilde{A} \in \mathcal{DO}_n(G)$, thus we can take $T, \Lambda$, $\Gamma$(and $\lambda$, if $G=\Sp(2d,\R)$) as described in Proposition \ref{DIAGNALIZATION LEMMA},  
 and $L_1(\Gamma)=L_{n+1}(A)=0$. By the reduction assumption, there exists $\tilde{\Gamma}$ which is so close to $\Gamma$ that $\tilde{\Gamma} \in \mathcal{R}(G)$ while $\tilde{\Gamma} \in \mathcal{UH}(G)$. Take $$A^\prime(x)=T(x+\alpha)\begin{pmatrix}
    \lambda(x)\Lambda(x) & O \\
    O & \lambda(x)^{-1}\Lambda(x)^{-T}
\end{pmatrix}\diamond \tilde{\Gamma}(x)T(x)^{-1},$$ 
for the case $G=\Sp(2d,\R)$, and if $G=\HSp(2d)$ we take
$$A^\prime(x)=T(x+\alpha)\begin{pmatrix}
    \Lambda(x) & O \\
    O & \Lambda(x)^{-*}
\end{pmatrix}\diamond \tilde{\Gamma}(x)T(x)^{-1},$$
then $A^\prime \in \mathcal{UH}(G)$ and is close to $A$.    
\qed

\section{Analytic Block diagonalization in symplectic  groups}\label{subsec: diag real}

\subsection{Holomorphic structure of $Gr(n,\C^d)$}
Recall that $Gr(n,\C^d)$, the Grassmannian
of $n$-dimensional subspaces of $\C^d$, which can be viewed as a complex homogeneous space $\GL(d,\C)/\GL(n,d)$, where $\GL(n,d):=\begin{pmatrix}
    A &C\\&D
\end{pmatrix}, A\in \GL(n,\C),D\in \GL(d-n,\C),C\in \C^{n\times (d-n)}$, with the standard holomorphic structure as a complex manifold. See Appendix of \cite{1D} for more details. In particular, we have the following corollary of Theorem A.1. (iii). of \cite{1D}.
\begin{lemma}\label{lem: holo map to complx Gr}A map $E$ from a region $D\subset \C$ to $Gr(n,\C^d)$ is holomorphic if for any $x\in D$, there is a neighborhood $N_x$ of $x$ in $D$ and holomorphic $\C^d$-valued mappings $v_1,\dots, v_n$ \footnote{i.e. each component of $v_i$ is a holomophic function.} defined on $N_x$ such that 
$E|_{N_x}=\mathrm{span}_\C\{v_1,\dots, v_n\}$.
\end{lemma}
The following definition follows the standard definition of real analytic mappings.
\begin{definition}\label{def: r-anal to c gr}Let $I$ be an open interval or $\R$, A map $E: I\to Gr(n,\C^d)$ is called real analytic if for any $x\in I$, there is a neighborhood $N_x\subset \C$ of $x$ and an extension $\tilde E: N_x\to Gr(n,\C^d)$ which is holomorphic and $\tilde E|_{N_x\cap I}=E|_{N_x\cap I}$. In particular in the view of Lemma \ref{lem: holo map to complx Gr}, $E$ is real analytic iff for any $x\in I$, there is a neighborhood $N_x\subset \C$ of $x$ and holomorphic $\C^d$-valued mappings $v_i, 1\leq i\leq n$ defined on $N_x$ such that $E|_{N_x\cap I}:=\mathrm{span}_\C\{v_1,\dots, v_n\}$.
\end{definition}

\subsection{Real-analytic  structure of $Gr(n,\R^d)$}


The fact we wish to clarify rigorously is: for $\alpha \in \R \backslash \Q$ and $A \in \mathcal{DO}_n(\Sp(2d,\R))$, its invariant sub-bundles $E^\ast, \ast\in \{s,u,c\}$, where $E^\ast(x)$ is a subspace of $\R^{2d}$, depend on $x \in \T$ analytically. 
First we build the $\R$-analogue of Lemma \ref{lem: holo map to complx Gr}.

Recall that for an $n$-dimensional real linear subspace $V\subset \R^d$, its complex extension $V^\C$ can be identified as $V\oplus i\cdot V\subset \C^d$, with an $n$-dimensional complex linear space structure and a $2n$-dimensional real linear space structure.

\begin{definition}\label{The exact meaning of analytical subspace}
\begin{enumerate}
    \item Let $I$ be an open interval or $\R$, A map $E: I\to Gr(n,\R^{d})$ is called real-analytic if its complex extension $$E^\C: I\to Gr(n,\C^{d}), E^\C(x):=E(x)\oplus i\cdot E(x), \forall x\in I,$$ is real-analytic in the sense of Definition \ref{def: r-anal to c gr}. More precisely, for any $x\in I$ there exists a neighborhood $N_x\subset \C$ of $x$ and an extension $\tilde E^\C: N_x\to Gr(n,\C^d) $ which is holomorphic \footnote{here $Gr(n,\C^d)$ has a canonical complex structure so the holomorphicity of $\tilde E^\C$ from a region in $\C$ to $Gr(n,\C^d)$ is well-defined.} and $\tilde E^\C|_{N_x\cap I}=E^\C$.
    \item A map $E: \T\to Gr(n,\R^{d})$ is called real-analytic if it can be canonically lifted to a real-analytic map $\tilde E:\R\to Gr(n,\R^{d})$ in the sense of (1) and satisfies $\tilde E(x+1)=\tilde E(x).$
\end{enumerate}

\end{definition}

\begin{remark}
    Since the lift $\tilde{E}$ in Definition \ref{The exact meaning of analytical subspace} is canonical,
    we identify $\tilde{E}$ and $E$ in the following just for convenience. 
\end{remark}

In the following, for $I$ being an open interval, $\T$ or $\R$, denote all the real-analytic map: $E: I\to Gr(n,\R^{d})$ by $C^\omega(I,Gr(n,\R^{d}))$. And as direct corollary of Definition \ref{The exact meaning of analytical subspace}, we have the following elementary observation:

\begin{corollary}
    Let $\alpha \in \R \backslash \Q$ and $A \in \mathcal{DO}_n(\Sp(2d,\R))$, then all the invariant real sub-bundles $E^\ast, \ast \in \{s,u,c\}$ of the cocycle $(\alpha,A)$ are real-analytic.
\end{corollary}
\begin{proof}
  Indeed one can  view $A$ as a complex cocycle and with invariant complex sub-bundles $(E^\ast)^\C$ which are exactly complex extension of $E^\ast$ respectively. According to Theorem \ref{HOLOMORPHIC SECTION} $(E^\ast)^\C$ are all holomorphic, hence by (1) of Definition \ref{The exact meaning of analytical subspace}, $E^\ast$ are real analytic.   
\end{proof}
 
 

In practice we usually verify real analyticity of invariant bundles by the following elementary lemma which is a real version of Lemma \ref{lem: holo map to complx Gr}:

\begin{lemma}\label{equi-def}
Let $1\leq n\leq d$, $I$ is an interval or $\R$. Then $E\in C^\omega(I,Gr(n,\R^{d}))$ if and only if for any $x \in I$, there exists a neighborhood $(x-\delta, x+\delta)$ and real analytic $\R^d$-valued functions $v_i, 1\leq i\leq n$ defined on $(x-\delta,x+\delta)$  such that we have $E|_{(x-\delta, x+\delta)}=\operatorname{span}_{\R}\{v_1,v_2,\dots,v_n\}.$    
\end{lemma}\label{from tildeE to E}

\begin{remark}\label{local}
    The choice of $v_i, 1\leq i\leq n$ depend on the choice of the base point $x$. In general, the frames $\{v_i\}$ might not be globally defined on $\T$. 
\end{remark}
\begin{proof}
The proof of "if" part follows directly from Definition \ref{def: r-anal to c gr}. In the following we show the "only if" part.

By Definition \ref{The exact meaning of analytical subspace}, Lemma \ref{lem: holo map to complx Gr} and Definition \ref{def: r-anal to c gr}, for any $x\in J$ there exists a neighborhood $(x-\delta, x+\delta)$ and holomorphic $\C^d$-valued mappings $\tilde v_1,\dots, \tilde v_n$, such that $$E^\C|_{(x-\delta,x+\delta)}=(E\oplus i\cdot E)|_{(x-\delta,x+\delta)}=\mathrm {span}_\C(\tilde v_1,\dots, \tilde v_n).$$
Since \( S := \{\Re \tilde{v_i}(x) : 1 \leq i \leq n\} \cup \{\Im \tilde{v_i}(x) : 1 \leq i \leq n\} \) spans \( E(x) \), and given that \( \dim_{\mathbb{R}} E(x) = n \), we can select \( n \) vectors \( \{v_i(x) : 1 \leq i \leq n\} \) from \( S \) that span \( E(x) \). Moreover, since \( \delta \) is sufficiently small, we know that \( \{v_i : 1 \leq i \leq n\} \) spans \( E \) over the interval \( (x-\delta, x+\delta) \).
We show that $\R^d$ valued functions $v_i:= \Re \tilde v_i, 1\leq i\leq n$ are exactly  what we want. 
\begin{itemize}
    \item Real analyticity of $v_i$ follows the general fact that the real part of holomorphic functions are real analytic.
    \item From the definition of $E^\C$, for any vector $v\in E^\C$ we have that $\Re v\in E$ and $\Im v\in E$. So $v_i$ takes values in $E$.
    \end{itemize}
 \end{proof}

Now we consider some useful properties for analytic sub-bundles.
Let $V$ be a subspace of $\mathbb{R}^{2d}$, we denote by $V^{\perp}$  the (Euclidean) orthogonal complement of $V$,  and denote $V^{\bot_{sp}}$ the 
symplectic-orthogonal complement of $V$: 
   \begin{center}
       $V^{\bot_{sp}}=\{w\in\mathbb{R}^{2d}: v^{T}Jw=0$ for all $v \in V\}$
   \end{center}   
respectively. Moreover, 
for any real Grassmannian-valued map $E\in C^0(I, Gr(n, d)), n\leq d$, we denote the pointwise defined orthogonal complement (resp. symplectic-orthogonal complement) of $E$ by $E^\perp$ (resp. $E^{\perp_{sp}}$) respectively.
Then we have 

\begin{lemma}\label{The analytical projction}
    Let $1 \leq n \leq d$, $E \in C^\omega(\T,Gr(n,\R^{d}))$, then 
    \begin{enumerate}
   \item for any $w \in C^\omega(\T,\R^{d})$, if we denote by $P_{E(x)}w(x)$ the orthogonal projection from $w(x)$ to $E(x)$, then $P_{E(x)}w(x)\in C^\omega(\T,\R^{d})$.
        \item $E^{\perp}$ $\in C^\omega(\T,Gr(d-n,\R^d))$. 
        \item Moreover, if  $E(x)$ is symplectic subspace, then $E^{\perp_{sp}} \in C^\omega(\T,Gr(d-n,\R^{d})).$ 
    \end{enumerate}
\end{lemma}
\begin{proof}

Since $ E \in C^\omega(\mathbb{T}, Gr(n, \mathbb{R}^d)) $, for any fixed $ x_0 \in \mathbb{T} $, we can apply Lemma \ref{equi-def} to select a $ \delta > 0 $ and vectors $ v_1, v_2, \ldots, v_n \in C^\omega((x_0 - \delta, x_0 + \delta), \mathbb{R}^d) $ such that
$$
E(x)|_{(x_0 - \delta, x_0 + \delta)} = \operatorname{span}_{\mathbb{R}}(v_1(x), v_2(x), \ldots, v_n(x))
$$
By Gram-Schmidt orthogonalization, we got a set of orthonormal basis vectors $ e_1, e_2, \ldots, e_n \in C^\omega((x_0 - \delta, x_0 + \delta), \mathbb{R}^d) $, and then
$$
P_{E(x)}w(x) = \sum_{j=1}^n \langle w(x), e_j(x) \rangle e_j(x),
$$
it follows that $ P_{E(\cdot)}w \in C^\omega((x_0 - \delta, x_0 + \delta), \mathbb{R}^d) $.
Notice that the definition of $ P_{E(x)}w(x) $, the unique vector $ v \in E(x) $ such that $ v - w(x) $ is orthogonal to $ E(x) $, is independent of the choice of the basis $ e_1, \ldots, e_n $. Therefore, it is well-defined for $ x \in \mathbb{R} $. Furthermore, since $ P_{E(x)}w(x) $ is periodic, this establishes the desired result in statement (1).

Let $ w_1, w_2, \ldots, w_d \in \mathbb{R}^{d} $ form a basis of $ \mathbb{R}^{d} $. Define $ v_j(x) = w_j - P_{E(x)}w_j $. By statement (1), we have $ P_{E(x)}w_j \in C^\omega(\mathbb{T}, \mathbb{R}^{d}) $.
As $\dim E(x)^{\perp}=d-n$, then for any $ x_0 \in \mathbb{T} $, there exists $\delta>0$, and $ v_{j_1}(x), v_{j_2}(x), \ldots, v_{j_{d-n}}(x) \in C^\omega((x_0 - \delta, x_0 + \delta), \mathbb{R}^d) $, such that 
$$
E(x)^{\perp} = \operatorname{span}_{\mathbb{R}}(v_{j_1}(x), v_{j_2}(x), \ldots, v_{j_{d-n}}(x))
$$
 Furthermore, since $ E^{\perp}(x + 1) = E^{\perp}(x) $, by  (2) of Definition \ref{The exact meaning of analytical subspace} and Lemma \ref{equi-def}, we conclude that
$
E^{\perp} \in C^\omega(\mathbb{T}, Gr(d-n, \mathbb{R}^{d})).
$

   If $E\in  C^\omega(\T,Gr(n,\R^{d}))$ is symplectic, then $JE(x):=\{Jv: v\in E(x)\}\in C^\omega(\T,Gr(n,$ $\R^{d})$. 
Therefore $E^{\perp_{sp}}=(JE)^{\perp} \in C^\omega
    (\T,Gr(d-n,\R^{d}))$ by $(2)$.
\end{proof}

\subsection{Periodic real-analytic basis.}
The key step to show Proposition \ref{DIAGNALIZATION LEMMA} is to construct an analytic family of symplectic basis for invariant bundles. 
By holomorphic structure of the Grassmannian manifold $Gr(n,\mathbb{C}^d)$, then any periodic holomorphic complex subspaces can be equipped with a periodic holomorphic complex basis:

    \begin{theorem}\cite{1D}\label{COMPLEX LIFT}
        An analytic function $u \in C^\omega(\mathbb{T},Gr(n,\mathbb{C}^d))$ can be lifted to a one-periodic holomorphic function $\Tilde{u}: \T_\delta \rightarrow M_n(d)$ such that for any $ z \in \mathbb{T}_\delta, u(z)$ is spanned on $\mathbb{C}$ by the $n$ column vectors of $\Tilde{u}(z)$, where  $\mathbb{T}_\delta=\{z \in \mathbb{C}: |\Im z| <\delta \}$ and 
    $$M_n(d)=\{M\in \mathbb{C}^{d\times k}:\operatorname{rank}(M)=n\}.$$
         In other words, there exists $u_1,u_2,\dots,u_n \in C^\omega(\mathbb{T}_{\delta},\mathbb{C}^d) $ such that for any 
         $z \in \mathbb{T}_\delta$,
         \begin{center}          $u(z)=\operatorname{span}_\mathbb{C}\{u_1(z),u_2(z),\dots,u_n(z)\}$.
         \end{center}
    \end{theorem}

 From Theorem \ref{COMPLEX LIFT}, we observe that in the case of the complex Grassmannian  $Gr(n,\C^d)$, it is possible to find a globally defined periodic holomorphic basis. In contrast, as noted in Remark \ref{local}, the basis in the real Grassmannian $Gr(n,\R^d)$ is typically defined only locally. However, we show it is possible to construct the basis globally if we allow one vector(-valued function) in the basis might be anti-periodic.
 

   \begin{proposition}\label{REAL LIFT}
       Let $E \in C^\omega(\mathbb{T},Gr(n,\R^d))$. There exists 
           $v_1 \in C^\omega(\mathbb{T},\mathbb{PR}^d),v_2,\dots,v_n \in C^\omega(\mathbb{T},\mathbb{R}^d)$
        such that for all $x \in \T$,
     $E(x)=\operatorname{span}_\mathbb{R}(v_1(x),v_2(x),\dots,v_n(x))$.  \end{proposition}
\begin{proof}We prove this by induction for $n=\dim E$. For $n=1$,
by Lemma \ref{equi-def}, for any $x$, there exists  
$\tilde{v}_1 \in C^{\omega}((x-\delta, x+\delta),\R^d ),$ 
 such that we have $E|_{(x-\delta, x+\delta)}=\operatorname{span}_{\R}\{\tilde{v}_1\}.$    
let $E(0)\cap \mathbb{S}^d=\{v_0,-v_0\}$, then by analytic continuation theorem,  
the following initial value problem
\begin{equation}\label{eq1}
\left\{
    \begin{array}{cc}
       & v(x)\in E(x) \qquad
         v(0)=v_0 \\
        & v \in C^\omega(\mathbb{R},\mathbb{R}^d)
    \end{array}
\right.
\end{equation}
has a unique solution, we denote it by $v_1$, and 
either $v_1(1)=-v_0=-v_1(0)$, or $v_1(1)=v_0=v_1(0)$. By a continuity argument, the former condition implies $v_1(x)=-v_1(x+1)$ and the latter implies $v_1$ is $1$-periodic, which completes the proof for $n=1$.


If $n \geq 2$, we have $\operatorname{dim}_{\R}(E(x)^\bot)=d-n \leq d-2$. Therefore by smoothness of $E$ on $x$, the union $\bigcup_{x \in \mathbb{T}}E(x)^\bot \neq \mathbb{R}^d$.  Take an arbitrary $w \in \mathbb{R}^d- \bigcup_{x \in \mathbb{T}}E(x)^\bot$ and denote the orthogonal projection of $w$ to $E(x)$ by $P_{E(x)}w$, then $P_{E(x)}w \neq 0$ for any $x \in \mathbb{T}$. 
By Lemma \ref{The analytical projction}, $P_{E(x)}w\in C^\omega(\T,\R^{d})$. We denote the orthogonal complement of $P_{E(x)}w$ in $E(x)$ by $E_1(x)$, then $E_1\in C^\omega(\T, Gr(n-1,\R^{d}))$. In fact, by Lemma \ref{The analytical projction}, $
(P_{E(\cdot)}w)^{\perp} \in C^\omega(\T,Gr(d-1,\R^d))$, since $E$ and $
(P_{E(x)}w))^{\perp}$ intersect transversely, their intersection $E_1$ is also analytic. By the assumption of induction we can take $v_1 \in C^\omega(\T,\mathbb{PR}^{d})$ and $v_2,\cdots,v_{n-1} \in C^\omega(\T,\R^{d})$ with 
 $E_1(x)=\operatorname{span}_{\R}\{v_1(x), v_2(x), \cdots, v_{n-1}(x)\}.$ Take $v_n(x)=P_{E(x)}w$ then $\{v_1,v_2,\cdots,v_n\}$ satisfies the requirements.
\end{proof}
   
\subsection{Periodic analytic symplectic basis.}The second step to prove Proposition \ref{DIAGNALIZATION LEMMA} is to apply Proposition \ref{REAL LIFT} and a symplectic orthogonal process to get symplectic basis, i.e. Proposition \ref{REAL SYMPLECTIC LIFT}. Before it, we state an easy fact.


      \begin{lemma}\label{TAULEMMA}
          Let $E \in C^\omega(\mathbb{T},Gr(n,\mathbb{R}^d))$ and  $v_1,v_2,\cdots,v_n \in C^\omega(\mathbb{T},\mathbb{PR}^d)$. If 
          $$E(x)=\operatorname{span}_\mathbb{R}(v_1(x),v_2(x),\cdots,v_n(x))$$ 
          for any $x \in \mathbb{T}$, then $\tau(E)=\prod_{k=1}^n\tau(v_k)$.
      \end{lemma}
      \begin{proof}
          This follows directly from 
          $v_1(x+1)\wedge\cdots\wedge v_n(x+1)=\prod_{k=1}^n\tau(v_k)v_1(x)\wedge\cdots\wedge v_n(x).$        
      \end{proof}
Proposition \ref{REAL LIFT} and Lemma \ref{TAULEMMA} implies that if  $\tau (E)=-1$, then one can always select $\tau(v_1)=-1$. 
   
   \begin{proposition}\label{REAL SYMPLECTIC LIFT} We have the following conclusions:
\begin{enumerate}
    \item Let $n<d$, $E \in C^\omega(\mathbb{T},Gr(2n,\R^{2d}))$ is a symplectic subspace. Then there exists  $ v_{\pm 1} ,v_{\pm 2},\dots,v_{\pm n}
\in C^\omega(\mathbb{T},\mathbb{R}^{2d})$
which form a symplectic basis of $E(x)$ for every $x\in \T$.

\item If $E(x)=E^+(x)\oplus E^-(x)$, where $E^{\pm} \in C^\omega(\T,Gr(n,\R^{2d}))$  are isotropic subspaces, then there exists $v_{\pm 1} \in C^\omega(\mathbb{T},\mathbb{PR}^{2d}),v_{\pm 2},\dots,v_{\pm n} \in C^\omega(\mathbb{T},\mathbb{R}^{2d}),$     
such that $\tau(v_1)=\tau(v_{-1})=\tau(E^+)=\tau(E^-)$ and for any $x \in \mathbb{T}$, the set
       \begin{center}
           $\{v_1(x),v_2(x),\dots,v_n(x),v_{-1}(x),v_{-2}(x),\dots,v_{-n}(x)\}$
       \end{center}
        is a symplectic basis of $E(x)$, with   
    $$E^+(x)=\operatorname{span}_{\mathbb{R}}(v_1(x),v_2(x),\dots,v_n(x)), \qquad 
       E^-(x)=\operatorname{span}_{\mathbb{R}}(v_{-1}(x),v_{-2}(x),\dots,v_{-n}(x)).$$
     \end{enumerate}    
   \end{proposition}

   \begin{proof}
We prove $(1)$  by induction. If $n=1$,
by Proposition \ref{REAL LIFT}, there exists $ \hat{v}_{1} \in C^{\omega}(\T,\mathbb{PR}^d)$ and $\hat{v}_{-1} \in C^{\omega}(\T,{\R}^d),$ such that $$E(x)=\operatorname{span}_{\R}(\hat{v_1}(x_0),\hat{v}_{-1}(x_0)).$$
We assert that $\tau(\hat{v}_1)=1$, hence it is a periodic $\R^d$-valued function. If it is not the case, $\hat{v}_{-1}(x+1)^tJ\hat{v}_1(x+1)=-\hat{v}_{-1}(x)^t J\hat{v}_1(x)$ implies that there exists $ x_0 \in \T$ such that $\hat{v}_{-1}(x_0)^t J\hat{v}_1(x_0)=0$. Hence $\hat{v}_1(x_0) \in E(x) \cap E(x)^{\bot_{sp}}=\{0\},$ which is a contradiction.
Let $v_1(x)=\frac{\hat{v}_{1}(x)}{\hat{v}_{-1}(x)^t J\hat{v}_{1}(x)}$ and $v_{-1}(x)=\hat{v}_{-1}(x)$ we get the case for $n=1$.

Assume the statement holds for 
$n-1$, 
now we consider some $E \in C^\omega(\T, {Gr}(2n, 2d))$. By Lemma \ref{The analytical projction}, $E^{\perp_{sp}} \in C^{\omega}(\T,Gr(2d-2n,\R^{2d}))$. By Proposition \ref{REAL LIFT}, there exist vectors $\hat{v}_1,\cdots,\hat{v}_{d},\hat{v}_{-1},\cdots,\hat{v}_{-d}\in C^{\omega}(\T, \mathbb{P} \mathbb{R}^{2d})$ such that for for all $j \neq 1,n+1$, $\hat{v}_j$ can be realized as an element in $C^\omega(\T, \R^{2d})$
such that 
\begin{equation}\label{span E(x)}    E(x)=\operatorname{span}_{\mathbb{R}}\left(\hat{v}_1(x), \cdots, \hat{v}_n(x), \hat{v}_{-1}(x), \cdots, \hat{v}_{-n}(x)\right),
\end{equation}
\begin{equation}\label{span E(x) perp}    E^{\perp_{sp}}(x)=\operatorname{span}_{\mathbb{R}}\left(\hat{v}_{n+1}(x), \cdots, \hat{v}_d(x), \hat{v}_{-(n+1)}(x), \cdots, \hat{v}_{-d}(x)\right).
\end{equation}
By the same argument as above, the non-degeneracy of the symplectic form restricted to $E$ and $E^{\perp_{sp}}$ ensures 
$\tau(\hat{v}_1)=1.$ 
and $\tau(\hat{v}_{n+1})=1.$ Hence all $\hat v_i$ could be realized as periodic functions. Therefore
there exist $ a_j \in C^\omega(\T,\R) (1 \leqslant \left|j\right| \leqslant d)$ such that 
\begin{equation}\label{comparision}
    J \hat{v}_n(x)=\sum_{1 \leqslant \left|j\right| \leqslant d} a_j(x) \hat{v}_j(x)
\end{equation}
for any $x \in \T$,
we set $\tilde{v}_{-n}(x)=\sum_{j=1}^{n} a_{-j}(x) \hat{v}_{-j}(x)+\sum_{j=1}^{n-1} a_j(x) \hat{v}_j(x)$, find that $\tilde{v}_{-n} \in C^\omega(\T,\R^{2d})$ and  $\tilde{v}_{-n}(x) \in E(x)$ for any $x \in \T$. Since $\hat{v}_j(\cdot)\in E^{\perp_{sp}}$ for $\left|j\right|>n$, we have 
\begin{equation}\label{Omega=Omega}
    \Omega(\tilde{v}_{-n}(x), \hat{v}_n(x))=\Omega(J\hat{v}_{n}(x), \hat{v}_n(x))=\left\| \hat{v}_n(x)\right\|^2>0.
\end{equation}
Since $E_1(x)=\operatorname{span}_{\R}\{\tilde{v}_{-n}(x),\hat{v}_n(x)\}$ is a symplectic subspace of $\R^{2d}$, by Lemma \ref{The analytical projction} it holds that $E_1^{\perp_{sp}}\cap E \in C^\omega(\T,Gr(2n-2,\R^{2d}))$, so we can find $v_{n},v_{-n} \in C^\omega(\T,\R^{d})$ forming a symplectic  basis of $E_1$ and $v_{1},\cdots,v_{n-1},v_{-1},\cdots,v_{-(n-1)} \in C^\omega(\T,\R^{d})$ forming a symplectic basis of $E_1^{\perp_{sp}}\cap E$ by the inductive assumption. Then it is not hard to see $\{v_{\pm 1},\dots ,v_{\pm n}\}$ form a symplectic basis of $E$, which completes the proof of (1).

We prove the item (2)  by a similar induction process. For $n=1$, by Proposition \ref{REAL LIFT} there exists $\hat{v}_1, \hat{v}_{-1} \in C^\omega(\T, \mathbb{P} \mathbb{R}^{2 d})$ such that for any $x \in \T$, $$E^{+}(x)=\operatorname{span}_{\mathbb{R}}\left(\hat{v}_1(x)\right),E^{-}(x)=\operatorname{span}_{\mathbb{R}}(\hat{v}_{-1}(x)).$$
Since $E^{-}(x) \oplus E^{+}(x)$ is a symplectic space, an assumption similar to the proof of $(1)$ yields
$ \hat{v}_1^{t}(x) J \hat{v}_{-1}(x) \neq 0$  for all $x \in \R.$
Then 
$$\hat{v}_1(x+1)^{t} J \hat{v}_{-1}(x+1)=\tau\left(\hat{v}_1\right) \tau\left(\hat{v}_{-1}\right) \hat{v}_1(x)^{t} J \hat{v}_{-1}(x),$$ implies that $ \tau\left(\hat{v}_1\right)=\tau\left(\hat{v}_{-1}\right).$
Set $v_1(x)=\frac{\hat{v}_{1}(x)}{\hat{v}_{-1}(x)^tJ\hat{v}_{1}(x)}$ and $v_{-1}(x)=\hat{v}_{-1}(x)$ confirming the case for $n=1$.

Assume (2) holds for $n-1$. 
Then by Proposition \ref{REAL LIFT} and  $(1)$, there exist vectors $\hat{v}_1,\cdots,\hat{v}_{d},\hat v_{-1},\cdots,\hat v_{-d}\in C^{\omega}(\T, \mathbb{P} \mathbb{R}^{2d})$ with $\tau(\hat v_j)=1$ for all $j \neq \pm1$, such that \eqref{span E(x)} and \eqref{span E(x) perp} holds, moreover
$$E^+(x)=\operatorname{span}_{\mathbb{R}}\left(\hat{v}_{1}(x), \cdots, \hat{v}_{n}(x)\right), \quad E^-(x)=\operatorname{span}_{\mathbb{R}}\left(  \hat{v}_{-1}(x), \cdots, \hat{v}_{-n}(x)\right).
$$
Then \eqref{comparision} still holds for some $a_j \in C^\omega(\R,\mathbb{PR}^1)$ with $$a_j(x+1)=\tau(a_j) \cdot a_j(x), \forall x\in \R,$$ and $\tau(a_j)=\tau(\hat v_j)$ \footnote{Here we actually consider the expansion for $J\hat{v}_n$ with respect to $2d$ directions $\R\cdot \hat{v}_j$ rather than with respect to the basis $\{\hat{v}_{\pm j}\}$. Since each $\R\cdot \hat{v}_j$ is periodic, each component $a_j\cdot \hat v_j$ is periodic even though each $\hat v_j$ may not be periodic.}. Set $\hat{v}_{-n}(x)=\sum_{j=1}^{n} a_{-j}(x) \hat{v}_{-j}(x)$ and find that $\hat{v}_{-n} \in C^\omega(\T,\R^{2d})$ and \eqref{Omega=Omega} still holds. 
Set $E_1^{+}=\operatorname{span}_{\R}\{\hat{v}_n\},E_1^{-}=\operatorname{span}_{\R}\{\hat{v}_{-n}\},E_1=E_1^+ \oplus E_1^-$,  one can easily verify $E_1^\pm$ are isotropic subspaces and $E_1$ is a $2$-dimensional symplectic subspace.
Consider $F^{\pm}:=E_1^{\perp_{sp}}\cap E^{\pm} \in C^\omega(\T,Gr(n-1,2d))$ respectively. By elementary arguments in linear algebra, $F^\pm$ are both $n-1$ dimensional isotropic subspaces and $F:=F^+ \oplus F^-$ is a $2n-2$-dimensional symplectic subspace. 

By the inductive assumption, we can take  $v_1,\cdots,v_{n-1},v_{-1},\cdots,v_{-(n-1)}$ in $ C^\omega(\T,\mathbb{PR}^{2d})$ forming a symplectic basis for $F$, with $\tau(v_j)=1$ except $j=\pm1$ 
 and $$F^{\pm}=\operatorname{span}\{v_{\pm1},\cdots,v_{\pm({n-1})}\}$$ respectively. Similarly, since $E_1=E_1^+ \oplus E_1^-$ is a symplectic subspace, and $E_1^{\pm} $ are isotropic subspaces with $\operatorname{dim}(E_1)=2$, by the case $n=1$ there exists a symplectic basis $\{v_n,v_{-n}\}$ for $E_1$ with that $E_1^{\pm}=\operatorname{span}_{\R}\{v_{\pm n}\}$ respectively. Note that  $\tau(E_1^{\pm})=1$ implies that $\tau(v_{\pm n})=1$. Using the fact that $E_1^+\perp_{sp} F^-, E_1^-\perp_{sp} F^+$, the set $\{v_1,\cdots,v_{n},v_{-1},\cdots,v_{-n}\}$ forms a symplectic basis of $E$ with the desired property. Hence we complete the proof of (2) for $n$.
\end{proof}

\subsection{Proof of Proposition \ref{DIAGNALIZATION LEMMA} }
Now we  finish the whole proof of Proposition \ref{DIAGNALIZATION LEMMA}.

If  $\tau(E^u)=\tau(E^s)=1$, by Lemma \ref{U,S,C symplectic properties} and Proposition \ref{REAL SYMPLECTIC LIFT}, we can take 
vectors $ v_{\pm 1},v_{\pm 2},\dots,v_{\pm d}\in C^\omega(\mathbb{T},\mathbb{R}^{2d})$, such that 
    \begin{center}
     $\{v_1(x),v_2(x),\cdots,v_n(x),v_{-1}(x),v_{-2}(x),\cdots,v_{-n}(x)\}$
\end{center}     
$$\{v_{n+1}(x),v_{n+2}(x),\cdots,v_d(x),v_{-n-1}(x),v_{-n-2}(x),\cdots,v_{-d}(x)\}$$
forms a symplectic basis of $E^u(x)\oplus E^s(x)$ and     $E^c(x)$ separately. Moreover, 
\begin{equation*}      E^u(x)=\operatorname{span}_{\R}(v_1(x),v_2(x),\cdots,v_n(x)),     \quad          E^s(x)=\operatorname{span}_{\R}(v_{-1}(x),v_{-2}(x),\cdots,v_{-n}(x)),
       \end{equation*}  
It follows that
     $$\{v_1(x),v_2(x),\cdots,v_d(x),v_{-1}(x),v_{-2}(x),\cdots,v_{-d}(x)\}$$
 forms a symplectic basis for $\R^{2d}$, and if we define the matrix 
\begin{equation}\label{DEFINITIONOF T}
    T(x)=\begin{pmatrix}
    v_1(x),\cdots,v_d(x),v_{-1}(x),\cdots,v_{-d}(x)
\end{pmatrix}
\end{equation}
then $T(x) \in C^\omega
(\T,\Sp(2d,\R))$. By the invariance under the dynamics, there exists $\Lambda_1 \in C^\omega(\T,\GL(n,\R))$, $\Lambda_2\in C^\omega(\T,\GL(n,\R))$ and $\Gamma \in C^\omega(\T,\Sp(2d-2n,\R))$
such that 
\begin{equation}\label{T(X+ALPA)}
T(x+\alpha)^{-1}A(x)T(x)=\begin{pmatrix}
    \Lambda_1(x) & O \\
    O & \Lambda_2(x)
\end{pmatrix}\diamond \Gamma(x).
\end{equation}
As $T \in C^\omega(\T,\Sp(2d,\R))$, which implies that $T(x+\alpha)^{-1}A(x)T(x)\in \Sp(2d,\R)$. Hence, $ \Lambda_2(x)= \Lambda_1^{-\top}(x)$,  $\Gamma(x)\in \Sp(2d-2n,\R)$. Now we take $\lambda(x)=\sqrt{\operatorname{det}(\Lambda_1(x))}$, $\Lambda(x)=\frac{1}{\lambda(x)}\Lambda_1(x)$,  and completes the proof of Proposition \ref{DIAGNALIZATION LEMMA} $(1)$.

If  $\tau(E^u)=\tau(E^s)=-1$, similar as above,  by Proposition \ref{REAL SYMPLECTIC LIFT}, we can take 
           $ v_{\pm 2},\dots,v_{\pm d} \in C^\omega(\mathbb{T},\mathbb{R}^{2d})$,
    and $v_{\pm 1} \in C^\omega(\T,\mathbb{PR}^{2d})$ with $\tau(v_1)=\tau(v_{-1})=-1$,  such that 
    $$\{v_1,\cdots,v_d,v_{-1},\cdots,v_{-d}\}$$ forms a symplectic basis for $\R^{2d}$ . Hence if we define the matrix $T(x)$ as in \eqref{DEFINITIONOF T},  
Then $T(x) \in \Sp(2d,\R)$ for all $x \in \T$, but $T(x+1)=T(x)\mathcal{P}_d$. By the invariance under the dynamics, there exists $\Lambda_1 \in C^\omega(2\T,\GL(n,\R))$, $\Lambda_2\in C^\omega(2\T,\GL(n,\R))$ and $\Gamma \in C^\omega(2\T,\Sp(2d-2n,\R))$
such that \eqref{T(X+ALPA)} holds.
From the fact $T(x)\in \Sp(2d,\R)$, we have $T(x+\alpha)^{-1}A(x)T(x)\in \Sp(2d,\R)$, hence $ \Lambda_2(x)= \Lambda_1^{-\top}(x)$,  $\Gamma(x)\in \Sp(2d-2n,\R)$. Since $T(x+1)=T(x)\mathcal{P}_d$, a direct computation shows that $\Lambda_1(x+1)=\mathcal{T}_n^{-1}\Lambda_1(x)\mathcal{T}_n$ and $\Gamma(x+1)=\Gamma(x)$.  Take $\lambda(x)=\sqrt{\operatorname{det}(\Lambda_1(x))}$, $\Lambda(x)=\frac{1}{\lambda(x)}\Lambda_1(x)$, one can check that $\lambda(x+1)=\lambda(x)$ and thus $\Lambda(x+1)=\mathcal{T}_n^{-1}\Lambda(x)\mathcal{T}_n$. Hence we completes the proof of Proposition \ref{DIAGNALIZATION LEMMA} $(2)$.\qed

\section{Analytic Block diagonalization in Hermitian-symplectic groups}\label{Canonical frame for the three bundles in hermitian-symplectic space}

\subsection{Analytic Sylvester Inertia Theorem}

As previously noted, the proof of analytic block diagonalization within Hermitian-symplectic cocycles relies on the analytic Sylvester Inertia Theorem. 
However, let's begin with Kato's  analytic diagonalization of Hermitian matrix-valued functions \cite{Kato}:

\begin{lemma}\cite{Kato}\label{Kato lemma}
    Let $I$ being on open interval on $\R$, $G \in C^\omega(I,\Her(m,\C))$. Then for  any $x_0 \in I$,  there exist $\delta>0$,  $U \in C^\omega_{\delta}((x_0-\delta,x_0+\delta),\U(m))$, $\lambda_i \in C^\omega_{\delta}((x_0-\delta,x_0+\delta),\R) $ such that whenever $\left|x-x_0\right|<\delta$,
    \begin{equation}\label{equation in kato lemma}
         U(x)^*G(x)U(x)=diag \{\lambda_1(x),\cdots,\lambda_{m}(x)\}.
    \end{equation}           
\end{lemma}

What we want to prove is the following: if $ G \in C^\omega(\T,\Her(m,\C))$ is periodic, $G(\cdot)$ can be analytically diagonalized by a  $U(\cdot)\in C^\omega(\R,\U(m))$, which is periodic up to a permutation matrix $\Gamma$.
The cyclic structure of $\Gamma$ reflects the monodromy of eigenvectors around the torus:

\begin{proposition} \label{permutation diagonalization}
    Let $ G \in C^\omega(\T,\Her(m,\C))$. Then there exists $U \in C^\omega(\R, \U(m))$ such that $U(x)^*G(x)U(x)$ is diagonal matrix for all $x \in \R$. 
Moreover, there exists $r_1,\cdots, r_s\in \N_+$ such that 
\begin{equation}\label{EIGEN II}
 \quad U(x+1)=U(x)\Gamma, \quad \text{where} \quad \Gamma=\operatorname{diag}(\Gamma_{r_1} \dots, \Gamma_{r_s}), \quad \Gamma_r:=\begin{pmatrix}
         0 & 1 \\
         I_{r-1} & 0
     \end{pmatrix}.
\end{equation}
\end{proposition}
\begin{proof}

We will distinguish the whole proof into several steps:

\smallskip

\textbf{Step 1: Permutation of analytic eigenvalues.}

For analytic $G(\cdot)$, eigenvalues and eigenvectors can be chosen analytically locally by Lemma \ref{Kato lemma}.
At points where eigenvalues cross (i.e. with multiplicities), eigenvalues may not be analytically defined. The analytic continuation around $\T$ leads to a permutation of eigenvalues.  We first prove that the eigenvalues of 
$G(\cdot)$ can be chosen as analytic functions on 
$\R.$

For convenience we denote the multiset composed by all the eigenvalues (with multiplicities taken into account) of $G(x)$ by $\Sigma(G(x))$. Then $G(x+1)=G(x)$ implies that for all $x \in \R$,
\begin{equation}\label{Sigma is periodic}
    \Sigma(G(x+1))=\Sigma(G(x)).
\end{equation}

\begin{lemma}\label{Eigenvalues}
    There exists $\delta>0$ and  $\lambda_1, \cdots,\lambda_m \in C_{\delta}^\omega(\R,\R)$, such that for all $x \in \R$,
    \begin{equation}\label{Sigma=Lambdas}
        \Sigma(G(x))=\{\lambda_1(x),\cdots,\lambda_m(x)\}.
    \end{equation}
\end{lemma}

\begin{proof}
Based on Lemma \ref{Kato lemma}, we can utilize compactness to select a positive $\delta > 0$ and a finite collection of open intervals $I_1, \ldots, I_{\bar{q}}$ such that $[0, 1] \subset \bigcup_{r=1}^{\bar{q}} I_r$. For each $r = 1, \ldots, \bar{q}$, the set $\Sigma(G)$ can be represented as follows, using functions $\lambda_{r,1}, \ldots, \lambda_{r,m} \in C_{\delta}^\omega(I_r, \mathbb{R})$:
\begin{equation}\label{Sigma=Lambdas0}
\Sigma(G(x)) = \{\lambda_{r,1}(x), \ldots, \lambda_{r,m}(x)\}, \quad \text{if} \quad x \in I_r.
\end{equation}
For any $k \in \mathbb{Z}$, we express $k = l \bar{q} + r$, where $l \in \mathbb{Z}$ and $1 \leq r \leq \bar{q}$. We define $I_k = I_{l \bar{q} + r} := l + I_r$, thereby establishing $\mathbb{R} = \bigcup_{k \in \mathbb{Z}} I_k$. Since $\Sigma(G)$ can be expressed as \eqref{Sigma=Lambdas0} for some functions $\lambda_{r,1}, \ldots, \lambda_{r,m} \in C_{\delta}^\omega(I_r, \mathbb{R})$ within $I_r$, we define $\lambda_{k,j} = \lambda_{l \bar{q} + r,j} := \lambda_{r,j}(\cdot + l)$. By using \eqref{Sigma is periodic}, we find that $\Sigma(G)$ can similarly be expressed as \eqref{Sigma=Lambdas0} for functions $\lambda_{k,1}, \ldots, \lambda_{k,m} \in C_{\delta}^\omega(I_k, \mathbb{R})$ within $I_k$.


    Next, we define  $\lambda_j \in C_{\delta}^\omega(\R,\R)$, $j=1,\cdots,m$  by an inductive anyalytic continuation argument such that  
    \eqref{Sigma=Lambdas} holds. 
Define $\lambda_j(x) = \lambda_{1,j}(x)$. Therefore, $\lambda_j \in C_{\delta}^\omega(I_1, \mathbb{R})$, and equation \eqref{Sigma=Lambdas} holds for all $x \in I_1$. Assume $\lambda_1, \ldots, \lambda_m$ is defined on $I_k$ for some $k \in \mathbb{N}_{+}$ and satisfies \eqref{Sigma=Lambdas}. We proceed to define the values they take on $I_{k+1}$.
By equation \eqref{Sigma=Lambdas0} and the inductive hypothesis, for all $x \in I_k \cap I_{k+1}$, it follows that $\{\lambda_1(x), \ldots, \lambda_m(x)\} = \{\lambda_{k,1}(x), \ldots, \lambda_{k,m}(x)\}$. Consequently, there exists a permutation $\sigma_x \in S_m$ such that $\lambda_j(x) = \lambda_{k, \sigma_x(j)}(x)$. Since $S_m$ is a finite set, there exists a fixed permutation $\sigma' \in S_m$ such that the set $\{x \in I_k \cap I_{k+1} : \sigma_x = \sigma'\}$ is not discrete. In other words, the zero set of $\lambda_j - \lambda_{k, \sigma'(j)}$ has an accumulation point. By the uniqueness property of analytic functions, we conclude that $\lambda_j = \lambda_{k, \sigma'(j)}$ holds identically in $I_k \cap I_{k+1}$. Thus, we define $\lambda_j(x) = \lambda_{k, \sigma'(j)}(x)$ for all $x \in I_{k+1}$, ensuring that $\lambda_j \in C_{\delta}^\omega(I_{k+1}, \mathbb{R})$ and that equation \eqref{Sigma=Lambdas} holds.

Continuing this inductive process, we can define the values of $\lambda_j$ on $\bigcup_{k \geq 1} I_k$. By a similar argument, $\lambda_j$ can also be defined on $\bigcup_{k \leq 0} I_k$. Since $\lambda_j \in C_{\delta}^\omega(I_k, \mathbb{R})$ for any $k \in \mathbb{Z}$, it follows that $\lambda_j \in C_{\delta}^\omega(\mathbb{R}, \mathbb{R})$. Thus, by the inductive step and equation \eqref{Sigma=Lambdas0}, we confirm that \eqref{Sigma=Lambdas} holds. Hence, we complete the proof of Lemma \ref{Eigenvalues}. 
\end{proof}

 By the uniqueness property of analytic functions, for $1 \leq j, k \leq m$, either $\lambda_j=\lambda_k$ holds identically, or $\lambda_j(x) \neq \lambda_k(x)$ holds for all $x \in \R$ except for only a discrete subset of $\R$.
    Up to a permutation among $\lambda_1,\cdots,\lambda_m$, we can take some $1=m_1<m_2\cdots<m_r<m_{r+1}=m+1$, such that
    \begin{equation}\label{Lambdas are equal}
        \lambda_{m_1}= \cdots = \lambda_{m_2-1}, \quad \lambda_{m_2}= \cdots = \lambda_{m_3-1}, \quad \cdots ,\quad \lambda_{m_{r}}= \cdots = \lambda_{m_{r+1}-1},
    \end{equation} holds identically, and for $1 \leq j,k \leq r $, $\lambda_{m_j}$ and $\lambda_{m_k}$ take different values out of a discrete subset of $\R$ whence $j \neq k$. Hence we can find a discrete subset of $\R$ (denoted by $S$) such that $\lambda_{m_1},\cdots,\lambda_{m_r}$ take $r$ distinct values out of $S$.

    Since $G(x+1)=G(x)$ implies  $\{\lambda_{m_1}(x+1),\cdots,\lambda_{m_r}(x+1)\}=\{\lambda_{m_1}(x),\cdots,\lambda_{m_r}(x)\},$ there exists a permutation $\tilde{\sigma} $ from $\{m_1,m_2,\cdots,m_r\}$ to $\{m_1,m_2,$ $\cdots,m_r\}$, such that $\lambda_{m_j}(\cdot+1)=\lambda_{\tilde{\sigma}(m_j)}$: Firstly we take some $x_0 \in \R\backslash S$, by \eqref{Sigma is periodic} and \eqref{Sigma=Lambdas}, we can find $\tilde{\sigma}$ satisfying $\lambda_{m_j}(x_0+1)=\lambda_{\tilde{\sigma}(m_j)}(x_0)$. Since $\{\lambda_{m_j}(x_0)\}$ are distinct we know $\lambda_{m_j}(x+1)=\lambda_{\tilde{\sigma}(m_j)}(x)$ holds for all $x \in (x_0-\delta,x_0+\delta)$ whence $\delta$ is small enough, hence holds for all $x \in \R$ by the uniqueness property of the analytic function.\\

\smallskip
\textbf{Step 2: Permutation of analytic eigenvectors.}

While the eigenvalues of 
$G(\cdot)$ can be defined globally and in an analytic way on
$\R$, the eigenvectors cannot generally be globally defined via analytic continuation due to the non-uniqueness of the unitary conjugacy $U(\cdot)$, 
 where $ U \in C_{\delta}^\omega((x-\delta,x+\delta), \mathrm{U}(m)) $ is obtained as in Lemma \ref{Kato lemma}. However, if we define
$$
E_{m_j}(x) = \operatorname{span}_{\mathbb{C}}(\operatorname{col}_{m_j} U(x), \ldots, \operatorname{col}_{m_{j+1}-1} U(x)),
$$
where $\operatorname{col}_{k} U(x)$ is $k-$th column of $U(x)$.  We can establish the following claim:

\begin{Claim} \label{Eigenspaces} For any $ j = 1, 2, \ldots, r $, we have: \begin{enumerate}
\item \( E_{m_j} \in C^\omega(\mathbb{R}, \mathrm{Gr}(m_{j+1} - m_j, \mathbb{C}^d)). \)
\item $
E_{m_j}(x + 1) = E_{\tilde{\sigma}(m_j)}(x).$
\end{enumerate}
\end{Claim}

\begin{proof} We begin by proving that $ E_{m_j} $ is well-defined and independent of the choice of $ U(\cdot) $. Indeed, for any $ \tilde{G}(\cdot) \in M(m, \mathbb{C}) $, we define
\[
\operatorname{Ker}(\tilde{G}(\cdot)) := \{ \xi \in \mathbb{C}^m : \tilde{G}(x)\xi = 0 \}.
\]
Therefore, for any \( x \in \mathbb{R} \setminus S \), \( E_{m_j}(x) = \operatorname{Ker}(\lambda_{m_j}(x) - G(x)) \) is independent of the choice of \( U \). It is important to note that if \( \operatorname{rank}(\tilde{G}(\cdot)) \) remains constant in a neighborhood of \( x_0 \), its kernel forms an analytic vector bundle. However, if \( \operatorname{rank}(\tilde{G}(\cdot)) \) drops at some \( x_0 \), the kernel dimension increases, leading to a loss of analyticity. Nonetheless, by Lemma \ref{Kato lemma}, \( U(x) \) depends analytically on \( x \in \mathbb{R} \), which implies that \( E_{m_j}(x) \) also depends analytically on \( x \). As a result, for \( x \in S \), \( E_{m_j}(x) \) corresponds to the limit of \( \operatorname{Ker}(\lambda_{m_j}(x_k) - G(x_k)) \), where the sequence \( \{ x_k \} \) in \( \mathbb{R} \setminus S \) converges to \( x \). This ensures that \( E_{m_j}(x) \) is independent of the choice of \( U \) as well.

If \( j \neq k \), then \( \lambda_{m_j}(x) \neq \lambda_{m_k}(x) \) implies that \( E_{m_j}(x) \) is orthogonal to \( E_{m_k}(x) \). By continuity, they remain mutually orthogonal for all \( x \in \mathbb{R} \). For any \( x \in \mathbb{R} \setminus S \), we find
\[
E_{m_j}(x + 1) = \operatorname{Ker}(\lambda_{m_j}(x + 1) - G(x + 1)) = \operatorname{Ker}(\lambda_{\tilde{\sigma}(m_j)}(x) - G(x)) = E_{\tilde{\sigma}(m_j)}(x),
\]
and by continuity, this holds for all \( x \in \mathbb{R} \).
\end{proof}

In the following, we will extend the permutation $\tilde{\sigma} $ defined on $\{m_1,m_2,\cdots,m_r\}$ to
a permutation $\sigma \in S_m$, and obtained the following:

\begin{lemma}\label{Eigenvectors}
There exist a permutation $\sigma \in S_m$ and $\xi_1,\cdots,\xi_m \in C^\omega(\R,\C^{m})$ forming an orthonormal frame, such that for  $k=1,2,\dots, m, x \in \R,$
    \begin{equation}\label{Eigenvector permutation}        G(x)\xi_k(x)=\lambda_k(x)\xi_k(x), \quad \lambda_k(x+1)=\lambda_{\sigma(k)}(x), \quad \xi_k(x+1)=\xi_{\sigma(k)}(x). 
    \end{equation}
\end{lemma}
\begin{proof}
Let $ T_k $ denote the smallest positive integer such that $ \tilde{\sigma}^{T_k}(m_k) = m_k $. We define the frame $ \{ \xi_{m_k}, \ldots, \xi_{m_{k+1}-1} \} $ for $ E_{m_k} \in C^\omega\left(T_k \mathbb{T}, \text{Gr}(m_{k+1}-m_k, \mathbb{C}^m)\right) $. To achieve this, we only  need to fix $ m_k $ such that
\begin{equation}\label{property of m_k}
    \min\{\tilde{\sigma}^j(m_k) : j \in \mathbb{Z}\} = m_k.
\end{equation}
We will first define the frame for $ E_{m_k} $, then subsequently define the frame for $ \{E_{\tilde{\sigma}^j(m_k)} : 1 \leq j \leq T_k-1\} $.

\textbf{Analytic frame for $ E_{m_k} $}:
According to Claim \ref{Eigenspaces}, it follows that $ E_{m_k} \in C^\omega(T_k \mathbb{T}, \text{Gr}(m_{k+1}-m_k, \mathbb{C}^m)) $. Therefore, by Theorem \ref{COMPLEX LIFT}, there exists a frame with period $ T_k $ for the bundle $ E_{m_k} $. After applying a Schmidt orthogonalization process, we obtain an orthonormal frame $ \{ \xi_{m_k}, \ldots, \xi_{m_{k+1}-1} \} $ for $ E_{m_k} $, where $ \xi_{m_k}, \ldots, \xi_{m_{k+1}-1} \in C^\omega\left(T_k \mathbb{T}, \mathbb{C}^m\right) $.

\textbf{Analytic frame for $ \{E_{\tilde{\sigma}^j(m_k)} : 1\leq j \leq T_k-1\} $}:
The key observation here is that if $ m_{k^\prime} = \tilde{\sigma}^j(m_k) $, then Claim \ref{Eigenspaces} implies that $ m_{k+1} - m_k = m_{k^\prime+1} - m_{k^\prime} $. Consequently, for $ 1 \leq q \leq m_{k+1} - m_k $, we define
 \begin{equation}\label{define of xi}
      \xi_{\tilde{\sigma}^j(m_k) + q}(x) = \xi_{m_k + q}(x + j).
 \end{equation}

Having established this, we  defined the entire analytic frame $ \{ \xi_1, \ldots, \xi_m \} $. Due to the orthonormality of $ \{ \xi_{m_k}, \ldots, \xi_{m_{k+1}-1} \} $, it follows that $ \{\xi_{\tilde{\sigma}^j(m_k) + q}: 1 \leq q \leq m_{k+1}-m_k\} $ forms an orthonormal frame for the bundle $ E_{\tilde{\sigma}^j(m_k)} $.
Since the spaces $ \{ E_{m_k} : 1 \leq k \leq r \} $ are pairwise orthogonal, we conclude that $ \{ \xi_1, \ldots, \xi_m \} $ forms an orthonormal frame.

Next, we extend the permutation $ \tilde{\sigma} $ to a permutation $ \sigma \in S_m $. For each $ 1 \leq k \leq r $ and $ 0 \leq q \leq m_{k+1}-m_k - 1 $, we define $ \sigma $ as follows:
\begin{equation}\label{DEfine sigma}
    \sigma(m_k + q) = \tilde{\sigma}(m_k) + q.
\end{equation}
For any $ l = 1, 2, \ldots, m $, we can represent $ l $ as $ l = m_k + q $ for some $ 1 \leq k \leq r $ and $ 0 \leq q \leq m_{k+1} - m_k - 1 $. Hence, we have
$$
\lambda_l(x + 1) = \lambda_{m_k + q}(x + 1) = \lambda_{m_k}(x + 1) = \lambda_{\tilde{\sigma}(m_k)}(x) = \lambda_{\tilde{\sigma}(m_k) + q}(x) = \lambda_{\sigma(l)}(x).
$$
Furthermore, we can assert the following:

\begin{Claim} For all $ 1 \leq l \leq m $, we have $ \xi_{l}(x+1) = \xi_{\sigma(l)}(x) $. \end{Claim}

\begin{proof} Let $ m_k $ be fixed such that \eqref{property of m_k} holds. It suffices to consider $ l $ as $ l = \sigma^j(m_k + q) $ with $ 0 \leq j \leq T_k-1 $ and $ 0 \leq q \leq m_{k+1} - m_k - 1 $. Since $ \sigma $ extends $ \tilde{\sigma} $, by the construction in \eqref{define of xi}, we have
$$
\xi_{{\sigma}^j(m_k + q)}(x) = \xi_{m_k + q}(x + j).
$$
\textbf{Case 1: $ 0 \leq j \leq T_k-2 $}: In this case, we find
$$
\xi_{{\sigma}^j(m_k + q)}(x+1) = \xi_{m_k + q}(x + 1 + j) = \xi_{{\sigma}^{j+1}(m_k + q)}(x),
$$ which implies
$
\xi_{l}(x+1) = \xi_{\sigma(l)}(x).
$\\
\textbf{Case 2: $ j = T_k-1 $}: Given that $ E_{m_k} \in C^\omega(T_k \mathbb{T}, \text{Gr}(m_{j+1} - m_j, \mathbb{C}^m)) $, it follows that $ \xi_{m_k + q}(x + T_k) = \xi_{m_k + q}(x) $. Since $ \sigma^{T_k}(m_k + q) = \tilde{\sigma}^{T_k}(m_k) + q = m_k + q $, we also have
$$
\xi_{{\sigma}^{T_k - 1}(m_k + q)}(x + 1) = \xi_{m_k + q}(x + T_k) =\xi_{m_k + q}(x )= \xi_{{\sigma}^{T_k}(m_k + q)}(x).
$$
Thus, the result follows. \end{proof}

Finally, since $ G(x)\xi = \lambda_{m_j}(x)\xi $ for any $ \xi \in E_{m_j}(x) $, we can conclude that
$
G(x)\xi_k(x) = \lambda_k(x)\xi_k(x)
$
for all $ k = 1, 2, \ldots, m $. We have therefore completed the proof.
\end{proof}

\smallskip
\textbf{Step 3: Constructing the unitary conjugacy $U(\cdot)$}

After traversing $\T$,  the permutation of eigenvectors is encoded by \eqref{Eigenvector permutation}. Then we decompose the permutation into cycles, and construct the unitary $U(\cdot)$ ensures consistency under the permutation. 

  We decompose $\sigma$ to be the composition of several disjoint cycles as $$\sigma=\left(a_1^1 \cdots a_{r_1}^1\right) \cdots \left(a_1^s \cdots a_{r_s}^s\right),$$ i.e. $\sigma(a_l^j)=a^j_{l+1}$ and $\sigma(a_{r_j}^j)=a^j_{1}$.  Denote $\lambda_l^j(x)=\lambda_{a^j_l}(x)$, $\xi_l^j(x)=\xi_{a^j_l}(x)$. Then by \eqref{Eigenvector permutation} and the fact that $\left\{\xi_1, \cdots, \xi_m\right\}$ is an orthonormal frame, we have the relationship  
    \begin{equation}\label{WHY U IS UNITARY} G_1(x)\xi^j_k(x)=\lambda^j_k(x)\xi^j_k(x), \quad \xi^j_k(x)^*\xi^{j^\prime}_{k^\prime}(x)=\delta_{jj^\prime} 
    \delta_{kk^\prime}=\left\{\begin{array}{cc}
         1, & j=j^\prime,k=k^\prime , \\
         0, & else.
     \end{array}\right.
        \end{equation}
        \begin{equation}\label{cyclic permutation}
         \lambda^j_k(x+1)=\lambda^j_{k+1}(x), \quad \xi^j_k(x+1)=\xi^j_{k+1}(x).
     \end{equation} 
holds for all $j \in\{1,2, \cdots, s\},$ $ k \in\left\{1,2, \cdots, r_j\right\}$ and $x \in \mathbb{T}$ (we identifies $\lambda_{r_j+1}^j$ and $\xi_{r_{j+1}}^j$ with $\lambda_1^j$ and $\xi_1^j$).
     For $j=1,2,\cdots,s$, we define the matrices $ U_j \in C^\omega(\R,\C^{n\times r_j})$ and $U \in C^\omega(\R,\C^{m \times m})$ as the following:
     $$U_j(x)=\begin{pmatrix}
         \xi^j_1(x),\cdots,\xi^j_{r_j}(x)
     \end{pmatrix}, \quad
U(x) = \begin{pmatrix}
U_1(x) & \cdots & U_s(x) 
\end{pmatrix}$$

By \eqref{WHY U IS UNITARY} we know  $U \in C^\omega(\T,\U(m))$. Since each column vector of $U(x)$ is an eigenvector of $G(x)$, we know $U(x)^*G(x)U(x)$ is a diagonal matrix, and \eqref{cyclic permutation} implies that $U_j(x+1)=U_j(x)\Gamma_{r_j}$, thus \eqref{EIGEN II} holds, completing the proof.
\end{proof}

\textbf{Proof for Proposition \ref{diagonalization}:}   
Take $U$ and $\Gamma$ as described in Proposition  \ref{permutation diagonalization}. 
Since $\Gamma_j^{r_j}=I_{r_j}$, we set $K=\prod_{j=1}^{s}r_j$ have $\Gamma^K=I_m,$ thus $U(x+K)=U(x)\Gamma^K=U(x)$ completes the proof of  Proposition \ref{diagonalization}.\qed
\bigskip

\textbf{Proof for Theorem \ref{diagonalization-1}:}  
Take $\lambda^j_k$ as stated in the proof of Proposition   \ref{permutation diagonalization}. Since $G \in C^\omega(\T, \GL(m,\C))$, for any $j=1,\cdots,s$ and $k \in \{1,\cdots,r_j\}$, by \eqref{WHY U IS UNITARY} we know $\lambda^j_{k+l}(\cdot)=\lambda_k^j(\cdot+l)$, 
$\lambda^j_k$ is either strictly positive or strictly negative, thus the notation $\operatorname{sign}(\lambda^j_k):=\operatorname{sign}\lambda^j_k(x)$ is well-defined, and it is independent of $k$ and depends solely on $j$, denote it by $\epsilon_j$. Then we define
$\Lambda_j \in C^\omega(\R,\R^{r_j \times r_j}), \Lambda \in C^\omega(\R,\R^{m \times m})$ by 
$$\Lambda_j(x)=\operatorname{diag}(\sqrt{\epsilon_j\lambda^j_1(x)},\cdots,\sqrt{\epsilon_j\lambda^j_{r_j}(x)}),\quad \Lambda(x) = \operatorname{diag}(\Lambda_1(x), \dots, \Lambda_s(x)).$$
Next we take $D=\operatorname{diag}(\epsilon_1I_{r_1},\cdots,\epsilon_sI_{r_s}).$ From \eqref{EIGEN II} we have $U \in C^\omega(\R,\U(m))$ and by definition,
\begin{equation}\label{UG1U=NORMFORM}
U(x)^*G(x)U(x)=\Lambda(x)^*D\Lambda(x).
\end{equation}

     For $r \in \N_+$, we define the matrices $Q_r \in C^\omega(\T,\U(r))$ as follows:
     \begin{equation}\label{def of unitary Q}
         Q_r(x)=\frac{1}{\sqrt{r}}\begin{pmatrix}
         e^{\frac{2\pi}{r}x} & e^{\frac{2\pi}{r}(x+1)} & \cdots &e^{\frac{2\pi}{r}(x+r-1)} \\
         e^{\frac{2\pi}{r}\cdot2x} & e^{\frac{2\pi}{r}\cdot2(x+1)} & \cdots &e^{\frac{2\pi}{r}\cdot2(x+r-1)} \\
         \vdots & \vdots &  & \vdots \\
         e^{\frac{2\pi}{r}\cdot rx} & e^{\frac{2\pi}{r}\cdot r (x+1)} & \cdots &e^{\frac{2\pi}{r}\cdot r(x+r-1)}
         \end{pmatrix}.
     \end{equation}
     Then we define
     $ Q(x)=\operatorname{diag}(Q_{r_1}(x) , \cdots , Q_{r_s}(x))$.
     Since $Q_r \in C^\omega(\R,\U(r))  $ we know $Q \in C^\omega(\R,\U(m))$. Moreover, one can verify that
     \begin{equation}\label{UandLambdaQ-PERMU}
Q_{r_j}(x+1)=Q_{r_j}(x)\Gamma_{r_j}, \quad Q(x+1)=Q(x)\Gamma.
     \end{equation}
Finally we take $M \in GL(m,\C)$ such that $M^*DM=\operatorname{diag}(I_p,-I_q)$ for some $p,q \in \N_+$ and define  $N(x)=U(x)\Lambda(x)^{-1}Q(x)^{*}M.$ By the proof of Lemma \ref{permutation diagonalization} we know 
$\Lambda(x+1)=\Gamma\Lambda(x)\Gamma,$ together with \eqref{EIGEN II}, \eqref{UG1U=NORMFORM} and \eqref{UandLambdaQ-PERMU} we know $N \in C^\omega(\T,\GL(m,\C))$.
Since $Q_j(x) \in U(r_j)$ impiles that $Q_{r_j}(x)I_{r_j}Q_{r_j}(x)^*=I_{r_j}$, we have $Q(x)DQ(x)^*=D$. By \eqref{UG1U=NORMFORM} we obtain the following 
$$N(x)^*G(x)N(x)=M^*Q(x)DQ(x)^*M=M^*DM=\operatorname{diag}(I_p,-I_q),$$
which completes the whole proof.
\qed

\subsection{Analytic Block diagonalization in Hermitian-symplectic groups}\label{abd}

Let's first consider the complexity of the problem from a view of Hermitian symplectic geometry. 
A remarkable fact is that there exist Hermitian-symplectic subspaces for which are impossible to identify a canonical basis, see \cite{Harmer}: 
let $V=\operatorname{span}_{\C}(v_1,v_2,\cdots,v_{2n})$ be a basis of 
the hermitian-subspace $V \subset \C^{2d}$, following M. Krein, 
for any $u,v \in V$, we define the  \textbf{Krein form}   by 
$i u^*Jv$ \cite{Long}, consequently we call 
\begin{equation}\label{GRAM MATRIX}
G(v_1,v_2,\cdots,v_{2n})=i 
    \begin{bmatrix}
        v_1^*Jv_1 & v_1^*Jv_2 & \cdots & v_1^*Jv_{2n} \\
        v_2^*Jv_1 & v_2^*Jv_2 & \cdots & v_2^*Jv_{2n} \\
        \vdots & \vdots & \ddots & \vdots \\
        v_{2n}^*Jv_1 & v_{2n}^*Jv_2 & \cdots & v_{2n}^*Jv_{2n} \\
    \end{bmatrix}
\end{equation}
the \textbf{Krein matrix} of this basis. 
It is straightforward to verify that $ G(v_1, v_2, \ldots, v_{2n}) $ constitutes a Hermitian matrix. If we choose an alternative basis $ \{u_1, u_2, \ldots, u_{2n}\} $ for the vector space $ V $, there exists a matrix $ N \in \GL(2n, \C) $ such that
$$
\begin{pmatrix}
v_1 & \cdots & v_{2n}
\end{pmatrix} = \begin{pmatrix}
u_1 & \cdots & u_{2n}
\end{pmatrix} N.
$$
A direct computation reveals that
$$
G(v_1, v_2, \ldots, v_{2n}) = N^*G(u_1, u_2, \ldots, u_{2n})N.
$$
This implies that the Krein matrices corresponding to the same Hermitian-symplectic subspace are congruent to one another.

Moreover, according to the Analytic Sylvester Inertia Theorem (Theorem \ref{diagonalization-1}), for any $ G \in \Her(2n, \C) \cap \GL(2n, \C) $, its congruence normal form is given by $ \operatorname{diag}(I_p, -I_q) $, where $ p $ represents the positive inertia index and $ q $ the negative inertia index of $ G $, with the condition that $ p + q = 2n $. It is clear that $ p $ and $ q $ are uniquely determined by $ p - q $, referred to as the signature difference of $ G $ and denoted by $ \operatorname{sign}(G) $. Thus, Krein matrices of the same Hermitian-symplectic subspace share the same congruence normal form characterized by their signature. It is noteworthy that there exists a matrix $ M \in \GL(2n, \C) $ such that
$$
M^* \begin{pmatrix}
I_n & O \\
O & -I_n
\end{pmatrix} M = \begin{pmatrix}
O & -iI_n \\
iI_n & O
\end{pmatrix}.
$$
Consequently, the vector space $ V $ possesses a canonical basis if and only if $ \operatorname{sign}(G) = 0 $.

Recall that $ A \in \mathcal{DO}_n(\HSp(2d)) $ implies $ E^s, E^u \in C^\omega(\T, Gr(n, \C^{2d})) $ and $ E^c \in C^\omega(\T, Gr(2d - 2n, \C^{2d})) $. By Lemma \ref{U,S,C symplectic properties}, the center space $ E^c $ is Hermitian-symplectic. Our key observation is that for the center space $ E^c $, the corresponding Klein matrix $ G $ satisfies $ \operatorname{sign}(G) = 0 $.

\begin{proposition}\label{Hermitian Lift}
\begin{enumerate} 
\item  There exists  $ v_{\pm 1} ,v_{\pm 2},\dots,v_{\pm n}
\in C^\omega(\mathbb{T},\mathbb{C}^{2d})$ which form 
 canonical basis of $E^s(x)\oplus E^u(x)$, moreover,  
           $$E^u(x)=\operatorname{span}_{\mathbb{C}}(v_1(x),v_2(x),\dots,v_n(x)), \quad E^s(x)=\operatorname{span}_{\mathbb{C}}(v_{-1}(x),v_{-2}(x),\dots,v_{-n}(x)).$$
       \item There exists $ v_{\pm (n+1)} ,v_{\pm (n+2)},\dots,v_{\pm d}
\in C^\omega(\mathbb{T},\mathbb{C}^{2d})$ which form 
 canonical basis of  $E^c(x)$.
 \end{enumerate}
\end{proposition}
\begin{proof}
       By Theorem \ref{COMPLEX LIFT}, there exists
$\hat{v}_1, \cdots, \hat{v}_d, \hat{v}_{-1}, \cdots, \hat{v}_{-d} \in C^{\omega}(\T, \mathbb{C}^{2d})$, such that 
$$E^{u}(x)=\operatorname{span}_{\C}\left(\hat{v}_1(x), \cdots, \hat{v}_n(x)\right),E^{s}(x)=\operatorname{span}_{\C}\left(\hat{v}_{-1}(x), \cdots, \hat{v}_{-n}(x)\right),$$
$$
E^c(x)=\operatorname{span}_{\C}(\hat{v}_{n+1}(x), \cdots, \hat{v}_d(x),\hat{v}_{-(n+1)}(x), \cdots, \hat{v}_{-d}(x)).
$$
By Lemma \ref{U,S,C symplectic properties}, $E^u(x),E^s(x)$ are Hemitian-isotropic subspaces, which implies there exists $H \in C^\omega(\T,\GL(n,\C))$ such that the Krein matrix $$G(\hat{v}_1(x),\cdots,\hat{v}_n(x),\hat{v}_{-1}(x),\cdots,\hat{v}_{-n}(x))=i \begin{pmatrix}
    O & -H(x) \\
    H(x)^* & O
\end{pmatrix}.$$
Hence if we take 
$$\begin{aligned}
    & \quad \begin{pmatrix}
    v_1(x) & \cdots & v_n(x) & v_{-1}(x) & \cdots & v_{-n}(x)
\end{pmatrix}\\&=\begin{pmatrix}
    \hat{v}_1(x) & \cdots & \hat{v}_n(x) & \hat{v}_{-1}(x) & \cdots & \hat{v}_{-n}(x)
\end{pmatrix}\operatorname{diag}(I_n,H(x)^{-1}),
\end{aligned},$$
then its corresponding Krein matrix satisfy 
$$G(v_1(x),\cdots,v_n(x), v_{-1}(x),\cdots, v_{-n}(x)) =
\begin{pmatrix}
    O & -iI_n \\
    iI_n & O \\
\end{pmatrix}$$
which implies desired result of $(1)$. 

To prove $(2)$, for any fixed $x \in \T$, let $k=\operatorname{sign}(G(\hat{v}_{n+1}(x),\cdots,\hat{v}_{d}(x),\hat{v}_{-n-1}(x),\cdots,\hat{v}_{-d}(x)))$. Then, there exists $N_x \in \GL(2n,\C)$ such that  $$G(\hat{v}_{n+1}(x),\cdots,\hat{v}_{d}(x),\hat{v}_{-(n+1)}(x),\cdots,\hat{v}_{-d}(x))=N_x^*\begin{pmatrix}
        I_{d-n+k} & O \\
        O & -I_{d-n-k} \\
    \end{pmatrix}N_x.$$    
    Thus the entire Krein matrix satisfies
    $$
    \begin{aligned}        & \quad G(v_1(x),\cdots,v_n(x),\hat{v}_{n+1}(x),\cdots,\hat{v}_{d}(x),v_{-1}(x),\cdots,v_{-n}(x),\hat{v}_{-(n+1)},\cdots,\hat{v}_{-d}(x))\\&=(N_x\diamond I_{d-n})^*\begin{pmatrix}
        I_{d+k} & O \\
        O & -I_{d-k} \\
    \end{pmatrix}(N_x\diamond I_{d-n}). 
    \end{aligned}
    $$    
    However, this matrix is congruent to  $J$, implying that $k=0$, then by Theorem \ref{diagonalization-1} we can find $N \in C^\omega(\T, \GL(2(d-n),\C))$ such that
$$N(x)^*G(\hat{v}_{n+1}(x),\cdots,\hat{v}_d(x),\hat{v}_{-(n+1)}(x),\cdots,\hat{v}_{-d}(x))N(x)=
\begin{pmatrix}
    I_{d-n}& O \\
    O & -I_{d-n} \\
\end{pmatrix},$$
then $(2)$ follows easily. 
\end{proof}

\bigskip

\subsection{Proof of Proposition \ref{DIAGNALIZATION LEMMA-1}}

We now complete the proof of Proposition \ref{DIAGNALIZATION LEMMA-1}. By Proposition \ref{Hermitian Lift}, we can select vectors $ v_{\pm 1}, v_{\pm 2}, \ldots, v_{\pm d} \in C^\omega(\mathbb{T}, \mathbb{C}^{2d}) $ such that for $ x \in \mathbb{T} $, we have:
$$
E^u(x) = \operatorname{span}_{\mathbb{C}}(v_1(x), v_2(x), \ldots, v_n(x)), \quad E^s(x) = \operatorname{span}_{\mathbb{C}}(v_{-1}(x), v_{-2}(x), \ldots, v_{-n}(x)),
$$
$$
E^c(x) = \operatorname{span}_{\mathbb{C}}(v_{n+1}(x), v_{n+2}(x), \ldots, v_d(x), v_{-n-1}(x), v_{-n-2}(x), \ldots, v_{-d}(x)).
$$
Furthermore, the set
$$
\{v_1(x), v_2(x), \ldots, v_d(x), v_{-1}(x), v_{-2}(x), \ldots, v_{-d}(x)\}
$$
forms a canonical basis for $\mathbb{C}^{2d}$. We define the matrix $ T(x) $ as in \eqref{DEFINITIONOF T}, which implies that $ T \in C^\omega(\mathbb{T}, \mathrm{HSp}(2d)) $ for all $ x \in \mathbb{T} $. By virtue of the invariance under dynamics, there exist $ \Lambda, \Lambda_1 \in C^\omega(\mathbb{T}, \mathrm{GL}(n, \mathbb{C})) $ and $ \Gamma \in C^\omega(\mathbb{T}, \mathrm{GL}(2d-2n)) $ such that \eqref{T(X+ALPA)} holds. The fact that $ T(x) \in \mathrm{HSp}(2d) $ implies that $ T(x+\alpha)^{-1}A(x)T(x) \in \mathrm{HSp}(2d) $. Hence, we conclude that $ \Lambda_1(x) = \Lambda^{-*}(x) $ and $ \Gamma(x) \in \mathrm{HSp}(2d-2) $, which completes the proof of Proposition \ref{DIAGNALIZATION LEMMA-1}.
\qed

\section{Proof of  Theorem \ref{SIMPLE DENSITY}}

In this section, we will try to focus on $\Sp(4,\R)$, and  give the full proof of Theorem \ref{SIMPLE DENSITY}. 
Firstly  we will prove a slightly weak result,  say the cocyles with simple Lyapunov spectrum  is dense in  $$ 
           \mathcal{R}(\Sp(4,\R))=\{A \in C^\omega(\mathbb{T},\Sp(4,\mathbb{R})):\omega^1(A)= \omega^2(A)=0 \}.
        $$

\subsection{A weaker result.}

\begin{theorem}\label{WEAK SIMPLE DENSITY}
        The cocycles $A$ with $L_1(A)>L_2(A)>0$ are dense in $\mathcal{R}(\Sp(4,\R)).$      
    \end{theorem}
\begin{proof}
By Theorem \ref{UH DENSITY}, we only need to prove that for any $A \in \mathcal{UH}(\Sp(4,\R))$, there exists $\tilde{A}\in \mathcal{R}(\Sp(4,\R))$ close to $A$ with $L_1(\tilde{A})>L_2(\tilde{A})>0$.
Of course, we only need to consider the case  $L_1(A)=L_2(A)>0$. 

As $A \in \mathcal{UH}(\Sp(4,\R))$, by Proposition \ref{DIAGNALIZATION LEMMA}, there exist $T \in C^\omega( \chi \T,\Sp(4,\R))$, $\lambda \in C^\omega(\T,\R_+)$,
    $\Lambda \in C^\omega(\chi \T,\SL(2,\R))$ where $\chi=1$ or $2$, such that  
   $$      T(x+\alpha)^{-1}A(x)T(x)=\begin{pmatrix}
    \lambda(x)\Lambda(x) & O \\
    O & \lambda(x)^{-1}\Lambda(x)^{-\top}
\end{pmatrix}.$$
Without lose of generality, we just assume $L(\lambda):=\int \ln|\lambda(x)|dx >0$. Thus one can easily check that 
 $L_1(A)=L_2(A)>0$ implies that $L(\Lambda)=0$ and $L_1(A)=L_2(A)=L(\lambda)$.
 If we denote  $$\omega (\lambda)= \lim\limits_{y\rightarrow 0^+}\frac{1}{2\pi y}(\int \ln|\lambda(x+iy)|dx-\int \ln|\lambda(x)|dx),$$
by Jensen's formula, it follows that $\omega (\lambda)\geq 0$, and $\omega (\lambda)\in \Z$. 
Thus $0=\omega_1(A)=\omega(\lambda)+\omega(\Lambda)$ implies that $\omega(\Lambda)=0.$

 If $\tau(E^s)=\tau(E^u)=1$, then $\chi=1$. By Theorem \ref{JMD POSITIVE DENSITY} we can take $\tilde{\Lambda} \in C^\omega
 (\T,\SL(2,\R))$ close to $\Lambda$ while $L(\tilde{\Lambda})>0$, moreover by the following Theorem,   \begin{theorem}\cite{1D}\label{CONTIOUS LYAPUNOV}
     The functions $\mathbb{R} \times C^\omega\left(\T, \GL(d,\C)\right) \ni(\alpha, A) \mapsto L_k(\alpha, A) \in$ $(-\infty, +\infty)$ are continuous at any $\left(\alpha^{\prime}, A^{\prime}\right)$ with $\alpha^{\prime} \in \mathbb{R} \backslash \mathbb{Q}$.
\end{theorem}
 we can take $\tilde{\Lambda}$ so close to $\Lambda$ that $L(\tilde{\Lambda})<L(\lambda).$ Then we set 
 \begin{equation}\label{TILDEA}
     \tilde{A}(x)=T(x+\alpha)\begin{pmatrix}
    \lambda(x)\tilde{\Lambda}(x) & O \\
    O & \lambda(x)^{-1}\tilde{\Lambda}(x)^{-\top}
\end{pmatrix}T(x)^{-1}.
 \end{equation}
Since $\mathcal{R}(\Sp(4,\R))$ is an open set, we can take $\tilde{A}$ so close to $A$ such that $\tilde{A}\in \mathcal{R}(\Sp(4,\R))$, and $L_1(\tilde{A})=L(\lambda)+L(\tilde{\Lambda})>L(\lambda)-L(\tilde{\Lambda})=L_2(\tilde{A})>0,$ which is just what we need.
 
If, instead,  $\tau(E^s)=\tau(E^u)=-1,$ then $\chi=2$. In the rest fo the papers, we set $\mathcal{T}=\mathcal{T}_2$ and $\mathcal{P}=\mathcal{P}_2$ fo for convenience. Then 
 we have $T(x+1)=T(x)\mathcal{P}$ and $\Lambda(x+1)=\mathcal{T}^{-1}\Lambda(x)\mathcal{T}$. In this case, denote 
$$\mathfrak{A}=\{\Lambda \in C^\omega(2\T,\SL(2,\R)): \Lambda(x+1)=\mathcal{T}^{-1}\Lambda(x)\mathcal{T}\}.$$
then the key observation is the following:
\begin{proposition}\label{ANTIPERIODIC POSITIVE DENSE}
    Let $\alpha \in \R\backslash\Q$. Then $\{\Lambda \in \mathfrak{A}:L(\Lambda)>0, \omega(\Lambda)=0\}$ is dense in $\{\Lambda \in \mathfrak{A}:\omega(\Lambda)=0\}.$     
\end{proposition}

For simplicity of the proof, if we denote 
\begin{equation}\label{DER ANTISELT}
    \mathscr{A}=\{\Lambda \in C^\omega(\T,\SL(2,\R)): \Lambda(x+\frac{1}{2})=\mathcal{T}^{-1}\Lambda(x)\mathcal{T}\},
\end{equation}
then Proposition \ref{ANTIPERIODIC POSITIVE DENSE} is equivalently to the following 
\begin{proposition}\label{ANTIPERIODIC POSITIVE DENSE-1}
    Let $\alpha \in \R\backslash\Q$. Then $\{\Lambda \in \mathscr{A}:L(\Lambda)>0, \omega(\Lambda)=0\}$ is dense in $\{\Lambda \in \mathscr{A}:\omega(\Lambda)=0\}.$     
\end{proposition} 

We left the proof of Proposition \ref{ANTIPERIODIC POSITIVE DENSE-1} (thus Proposition \ref{ANTIPERIODIC POSITIVE DENSE}) until  the last section. We first finish the proof of Theorem \ref{WEAK SIMPLE DENSITY} given Proposition \ref{ANTIPERIODIC POSITIVE DENSE}.
By Proposition \ref{ANTIPERIODIC POSITIVE DENSE} we can take $\tilde{\Lambda} \in \mathfrak{A}$ close to $\Lambda$ while $L(\tilde{\Lambda})>0$, and by Theorem \ref{CONTIOUS LYAPUNOV} we can take $\tilde{\Lambda}$ so close to $\Lambda$ that $L(\tilde{\Lambda})<L(\lambda).$

Then we set $\tilde{A}$ as in \eqref{TILDEA}, and we can verify that $\tilde{A}\in C^\omega(\T, \Sp(4,\R))$. Since $\mathcal{R}(\Sp(4,\R))$ is an open set, we can take $\tilde{A}$ so close to $A$ such that $\tilde{A}\in \mathcal{R}(\Sp(4,\R))$, and $L_1(\tilde{A})=L(\lambda)+L(\tilde{\Lambda})>L(\lambda)-L(\tilde{\Lambda})=L_2(\tilde{A})>0,$ which finishes the whole proof.
\end{proof}
\subsection{Proof of Theorem \ref{SIMPLE DENSITY}}

We proceed to demonstrate Theorem \ref{SIMPLE DENSITY} by leveraging Theorem \ref{WEAK SIMPLE DENSITY}. The crux of the proof of Theorem \ref{WEAK SIMPLE DENSITY} hinges on Proposition \ref{ANTIPERIODIC POSITIVE DENSE}, which asserts that uniformly hyperbolic cocycles are dense within the class of regular cocycles that maintain the group structure. The primary obstacle in proving Theorem \ref{SIMPLE DENSITY} is the possibility of positive acceleration, or more specifically, when the cocycle $(\alpha, \Lambda)$ is critical.  Our objective is to establish that critical behavior is still rare in $\mathscr{A}$. Indeed, it is feasible to extend the approach in \cite{Global} to show that critical cocycles are contained within a countable union of codimension-one analytic submanifolds of $\mathscr{A}$.  Here,
 the following result is sufficient for us:
\begin{proposition}\label{RARECRITICAL}
    $\{A \in  \mathscr{A}: \omega(A)>0,L(A)=0\}$ has no interior in $ \mathscr{A}$.
\end{proposition}

To prove Proposition \ref{RARECRITICAL}, the elementary observation is  the following: 
\begin{lemma}\label{DEGREE=0}
    For any $A \in \mathscr{A}$, we have $\operatorname{deg}(A)=0$.
\end{lemma}
\begin{proof}
 Since $A$ is homotopic to $R_{lx}$ for some $l \in \mathbb{Z}$, it follows that $A\left(x + \tfrac{1}{2}\right)$ is homotopic to $R_{l\left(x + \frac{1}{2}\right)}$, and $\mathcal{T}^{-1} A(x) \mathcal{T}$
    is homotopic to $R_{-lx}$. 
    Since $A \in \mathscr{A}$, we see that  $R_{l\left(x + \frac{1}{2}\right)}$ is homotopic to $R_{-lx}$, which implies $l = 0$. The result follows.
\end{proof}

Moreover, one important fact is that holomorphic extension  of the quasiperiodic cocycle is a regular cocycle, which was proved in \cite{1D}:
\begin{theorem}\label{thm-ii}\cite{1D}
     Let $\alpha \in \mathbb{R} \backslash \mathbb{Q}, A \in C^\omega\left(\T, \GL(d,\C) \right)$. Then for every $t \neq 0$ small enough, $(\alpha, A(\cdot+i t))$ is $k$-regular.
\end{theorem}

Once we have these, we can now finish the proof of  Proposition \ref{RARECRITICAL}:  

\begin{proof}  
We prove this by contradiction. Assume $\{A \in \mathscr{A}: \omega(A)>0,L(A)=0\}$ includes an open set $\mathcal{U}\in \mathscr{A}$, take $\omega_*=\min_{A\in \mathcal{U}}\omega(A)>0$. By the upper-semicontinuity of the acceleration, $\mathcal{V}=\{A \in \mathcal{U}: \omega(A)=\omega_*, L(A)=0\}$ is an open set of $\mathscr{A}$.

Fix some $A\in \mathcal{V}$. 
By Theorem \ref{thm-ii} and Theorem \ref{DOMINATE=REGULAR}, there exists $\epsilon$ such that for any $y \in (0,\epsilon)$, $A_{iy}:=A(\cdot+iy) \in \mathcal{UH}(\SL(2,\C))$, and therefore the Lyapunov of $(\alpha,A(\cdot+iy) )$ is differentiable in neighborhood of $\mathscr{A}$.  Indeed, for any $w \in C^\omega
(\T,sl(2,\R))$ with the property
$w(x+\frac{1}{2})=\mathcal{T}^{-1}w(x)\mathcal{T},$
$A e^{tw} \in \mathcal{U}$ for any $t$ small enough.  
By the assumption, 
\begin{equation}\label{L(Lambda)}
L(A(\cdot+iy)e^{tw(\cdot+iy)})=L(A e^{tw})+2\pi \omega_* y =2\pi \omega_* y
\end{equation}
for any $t$ small enough. We take derivative with respect to $t$ in 
\eqref{L(Lambda)} to get 
\begin{equation}\label{L(Lambd_+1}
\frac{d}{d t}  L(A(\cdot+iy)e^{tw(\cdot+iy)})=0, \quad \text{ at } t=0.
\end{equation}

Next we estimate  the coefficients of derivative of the Lyapunov exponent. To do this, note for any  $A \in C^\omega(\T,\SL(2,\C))$ such that $(\alpha,A)$ is uniformly hyperbolic. Let $u, s: \T \rightarrow \mathbb{P C}^2$ be the unstable and stable directions. Also let $B: \T \rightarrow \mathrm{\SL}(2, \mathbb{C})$ be analytic with column vectors in the directions of $u(x)$ and $s(x)$. Then there exists $\lambda \in C^\omega(\T,\C-\{0\})$ and $D \in C^\omega(\T, \SL(2,\C))$ such that  
\begin{equation}
    B(x+\alpha)^{-1} A(x) B(x)=\left(\begin{array}{cc}
\lambda(x) & 0 \\
0 & \lambda(x)^{-1}
\end{array}\right)=D(x) .
\end{equation}
Obviously $L(A)=L(D)=\int_{\T} \operatorname{Re} \log \lambda(x) d x.$
Write
\begin{equation}\label{B(x)}
    B(x)=\left(\begin{array}{ll}
a(x) & b(x) \\
c(x) & d(x)
\end{array}\right) .
\end{equation}
and denote 
\begin{equation}\label{q1(x)}
    q_1(x)=a(x) d(x)+b(x) c(x), \quad q_2(x)=c(x) d(x) \quad \text { and } \quad q_3(x)=-b(x) a(x)
\end{equation}
Actually, $q_i, i=1,2,3$ are the coefficients of the derivative of the Lyapunov exponent, which is clear  from the next lemma:
 \begin{lemma}\label{LEMMA9}\cite{Global}
     Let $(\alpha, A) \in \mathcal{U H}$ and let $q_1, q_2, q_3: \T \rightarrow \mathbb{C}$ be the coefficients of the derivative of the Lyapunov exponent. Let $w: \T \rightarrow \mathfrak{s l}(2, \mathbb{C})$ be analytic, and write
$$
w=\left(\begin{array}{cc}
w_1 & w_2 \\
w_3 & -w_1
\end{array}\right)
$$
Then one has
$$
\frac{d}{d t} L\left( A e^{t w}\right)=\operatorname{Re} \int_{\T} \sum_{i=1}^3 q_i(x) w_i(x) d x, \quad \text { at } t=0 .
$$
 \end{lemma}

The idea here is that if we further assume $A \in \mathscr{A}$, the conjugacy $B$ will inherite this group structure, this will imply   
the coefficients of derivative of the Lyapunov exponent will have some special property. We start with the following:

\begin{lemma}
  Let  $A \in \mathcal{UH}(\SL(2, \C))$ with that  $A(x+\frac{1}{2})=\mathcal{T}^{-1}A(x)\mathcal{T}$. Then there exist $\tilde{B} \in C^\omega(\T,\SL(2,\C))$, $\tilde{\lambda} \in C^\omega(\T,\C- \{0\})$ such that
  \begin{eqnarray}\label{B(x)1}
    \tilde{B}(x+\frac{1}{2})&=&\mathcal{T}^{-1}\tilde{B}(x)\mathcal{T},\\
    \label{conju}
    \tilde{B}(x+\alpha)^{-1} A(x) \tilde{B}(x)&=&\left(\begin{array}{cc}
\tilde{\lambda}(x) & 0 \\
0 & \tilde{\lambda} (x)^{-1}
\end{array}\right).
\end{eqnarray}
\end{lemma}
    
\begin{proof}
    By Theorem \ref{COMPLEX LIFT} we can take $u,s \in C^\omega(\T,\C^2)$  with    \begin{equation}\label{us}u(x)^{\top}J^{\top} s(x)=1, \qquad x\in \T,
    \end{equation} such that 
\begin{equation}\label{A(x)u(x)}        A(x)u(x)=\lambda(x)u(x+\alpha), \quad A(x)s(x)=\lambda(x)^{-1}s(x+\alpha).
\end{equation}
Combining \eqref{A(x)u(x)} and the fact $A(x+\frac{1}{2})=\mathcal{T}^{-1}A(x)\mathcal{T}$ we get 
$$
A(x)\mathcal{T}u(x+\frac{1}{2})=\lambda(x+\frac{1}{2}) \mathcal{T} u(x+\frac{1}{2}+\alpha),
$$
$$
A(x)\mathcal{T}s(x+\frac{1}{2})=\lambda(x+\frac{1}{2})^{-1}\mathcal{T} s(x+\frac{1}{2}+\alpha).
$$
By the uniqueness of the stable and unstable direction,  there exists $\zeta, \eta \in C^\omega(\T, \C)$ such that
\begin{equation}\label{U ANTISELF}
    \mathcal{T}u(x+\frac{1}{2})=\zeta(x)u(x), \qquad \mathcal{T}s(x+\frac{1}{2})=\eta(x)s(x).
\end{equation}
 Employing equation \eqref{us}, we deduce that $(\mathcal{T}u(x))^{\top}J^{\top}\mathcal{T}s(x)=-1$, which in turn implies $\eta(x)=-\zeta(x)^{-1}$. 
This just means if we take  $B(x)=\begin{pmatrix}
u(x), & s(x)    
\end{pmatrix}$, then it holds that 
\begin{eqnarray}\label{con1}
    \mathcal{T}B(x+\frac{1}{2})&=&B(x)\operatorname{diag}(\zeta(x), -\zeta(x)^{-1}),\\
  \label{con2}  B(x+\alpha)^{-1} A(x) B(x)&=&\left(\begin{array}{cc}
\lambda(x) & 0 \\
0 & \lambda(x)^{-1}
\end{array}\right)
    \end{eqnarray}
    The pivotal observation is as follows:
\begin{Claim} There exists a function $\xi \in C^\omega(\T,\C- \{0\})$ such that $\xi(x)\xi(x+\frac{1}{2})^{-1}=-\zeta(x)$. \end{Claim}
\begin{proof} Given that $u(x+1)=u(x)$, equation \eqref{U ANTISELF} implies $\mathcal{T}u(x)=\zeta(x+\frac{1}{2})u(x+\frac{1}{2})$, which further implies that \begin{equation}\label{ze1}
\zeta(x)\zeta(x+\frac{1}{2})=1.
\end{equation}
Since $\zeta \in C^\omega(\T, \C)$, there exists a function $\theta \in C^\omega(\R, \R)$ with the property that $\theta(x+1)-\theta(x) \in \mathbb{Z}$ for all $x \in \R$, such that $\zeta(x)=\left|\zeta(x)\right|e^{2\pi \operatorname{i}\theta(x)}$. From \eqref{ze1}, there exists a $K \in \mathbb{Z}$ for which $\theta(x+\frac{1}{2})+\theta(x)=K$ for all $x \in \R$. Let $\varphi(x)=\frac{1}{2}\theta(x)-(K+1)x$, then $\varphi(x)-\varphi(x+\frac{1}{2})=\theta(x)+\frac{1}{2}$, thus $\xi(x)=\left|\zeta(x)\right|^{\frac{1}{2}}e^{2\pi i\varphi(x)}$ satisfies our needs. \end{proof}

Once we have this, by noting \eqref{con1} and \eqref{con2}, then  $\tilde{B}(x)=B(x)\operatorname{diag}(\xi(x),\xi (x)^{-1 })$ will satisfy \eqref{B(x)1} and \eqref{conju}.
\end{proof}

By Theorem \ref{HOLOMORPHIC SECTION} and the fact that the coefficients of derivative of the Lyapunov exponent only depend on $(\alpha,A)$, we can find $q_1,q_2,q_3 \in C^\omega(\R_{\epsilon}\cap\mathbb{H}, \C)$ such that $q_1(\cdot+iy),q_2(\cdot+iy),q_3(\cdot+iy)$ is the derivative of the Lyapunov exponent of $(\alpha, A_{iy} e^{tw_{iy}})$. Since $A \in \mathscr{A}$, for any $y>0$ small enough it holds that $A_{iy}(x+\frac{1}{2})=\mathcal{T}^{-1}A_{iy}(x)\mathcal{T}$. If we 
write $B$ as in \eqref{B(x)}, then as consequence of \eqref{B(x)1}, we have
$$a(x+\frac{1}{2})=a(x), \quad b(x+\frac{1}{2})=-b(x),\quad
c(x+\frac{1}{2})=-c(x),\quad
d(x+\frac{1}{2})=d(x).
$$
which directly imply 
  \begin{equation}\label{q_1(x)1}
    q_1(x+iy+\frac{1}{2})=q_1(x+iy),\quad q_2(x+iy+\frac{1}{2})=-q_2(x+iy),\quad q_3(x+iy+\frac{1}{2})=-q_3(x+iy).
\end{equation}

By Lemma \ref{LEMMA9}, for any $w_1 \in C^\omega(\T,\R)$ and $w_2, w_3$ in $C^\omega(\T,\R)$ with the property $w_2(x+\frac{1}{2})=-w_2(x+\frac{1}{2}), w_3(x+\frac{1}{2})=-w_3(x+\frac{1}{2})$ we have the following identity:
$$\Re \int_{\T} \sum_{j=1}^3 q_i(x+iy) w_j(x+iy) d x=0.$$ First we take $w_1$ and $w_2$ to be identically-zero and then $w_1$ and $w_3$ to be identically-zero, we see that for any $w\in C^\omega(\T,\R)$ with the property $w(x+\frac{1}{2})=-w(x+\frac{1}{2}).$
\begin{equation}\label{Re}
    \Re \int_{\T}  q_j(x+iy) w(x+iy) d x=0, \quad j=2,3.
\end{equation}

Notice that $q_j$ is not defined on $\R$, but the number  
$$\hat{q_j}(k)=\int_0^1q(x+iy)e^{2k\pi i(x+iy)}$$
is well-defined for any $y \in (0,\epsilon)$, independent of the choice of $y$, and we still call it the $k$-th Fourier coefficient of $q_j \quad (j=2,3)$ . By \eqref{q_1(x)1} we see that  $\hat{q}(2k)=0$ for any $k \in \mathbb{Z}$.
In \eqref{Re}, take $w$ to be the form $a \cos (2k+1) \pi  x+b \sin (2k+1) \pi k x$, $a,b
\in \R$, $k \in \Z$, we see that the $\hat{q}(2k+1)=-\overline{\hat{q}(-2k-1)}$ for any $k \in \mathbb{N}_+$.  Since the Fourier series 
$\sum_{k\in \Z}\hat{q}(k)e^{2k\pi iz}$ converges whenever $0<\Im
z<\epsilon_0$, this implies that it actually converges whenever $0<\Im
z<\epsilon_0$, and at $\R$ it defines a purely imaginary function. Thus $q_j(x)$ extends analytically through $\R$ for $j=2,3$. 

The rest of argument is the same as \cite{Global}. Identifying $\mathbb{P} \C^2$ with the Riemann sphere in the usual way (the line
through $\begin{pmatrix}
    z\\w
\end{pmatrix}$ corresponding to $z/w$), and let $E^u(x)$, $E^s(x)$ be identified with $\phi_u(x)$, $\phi_s(x)$ under this identification. we get
$q_2=\frac {1} {\phi_u-\phi_s}$ and $q_3=\frac {\phi_u\phi_s} {\phi_u-\phi_s}$. These formulas allow us
to analytically continuate $\phi_u$ and $\phi_s$ through $\Im x=0$.  Since $q_2$ and
$q_3$ are purely imaginary at $\Im x=0$, we conclude that $\phi_u$ and $\phi_s$ are
complex conjugate directions in $\mathbb{P} \C^2$, and since they are distinct they
are also non-real. Consequently, we can find $T \in C^\omega(\T,\SL(2,\R))$ and $R \in C^\omega(\T,\SO(2,\R))$ such that $T(x+\alpha)^{-1}A(x)T(x)=R(x)$. Therefore, $\omega(A)>0$ implies  that $\omega(R)>0$ and then $\operatorname{deg}(R)>0$, which contradict to Lemma \ref{DEGREE=0}.
\end{proof}

Now we can finish the whole proof of Theorem \ref{SIMPLE DENSITY}. 
\begin{proof}
By Theorem \ref{POSITIVE DENSITY}, we can choose $A \in C^\omega(\mathbb{T},\Sp(4,\R))$ with $L_1(A)>0$ and $\omega^1(A)\omega^2(A)=0$. We can then divide the proof into four cases:

\smallskip \textbf{Case I: $\omega_1(A)=0$ and $L_1(A)>L_2(A)=0$.}

By Theorem \ref{DOMINATE=REGULAR}, $(\alpha,A)$ is $1,3-$dominated. Proposition \ref{DIAGNALIZATION LEMMA} implies the existence of $T \in C^\omega(\chi\T,\Sp(2d,\R))$, $\lambda \in C^\omega(\T,\R_+)$, and $\Gamma \in C^\omega(\chi\T,\SL(2,\R))$ such that $$T(x+\alpha)^{-1}A(x)T(x)=\operatorname{diag}(\lambda(x),\lambda(x)^{-1})
\diamond \Gamma(x).$$
Using Theorem \ref{JMD POSITIVE DENSITY}, we can find $\tilde{\Gamma}$ close to $\Gamma$ with $L(\tilde{\Gamma})>0$. Define $$\tilde{A}(x)=T(x+\alpha)\operatorname{diag}(\lambda(x),\lambda(x)^{-1})
\diamond \tilde{\Gamma}(x)T(x)^{-1}$$
Then $\tilde{A} \in C^\omega(\T, \Sp(4,\R))$ is close to $A$ and satisfies $L_1(\tilde{A})>L_2(\tilde{A})>0$.

\smallskip \textbf{Case II: $\omega_1(A)=0$ and $L_1(A)=L_2(A)>0$.}

Since $L_1(A(\cdot +it))\geq L_2(A(\cdot +it))$ for all $t$ in a neighborhood of $0$, and $L_1(A)=L_2(A)>0$, it follows that $\omega_2(A)\leq \omega_1(A)=0$. Thus, $\omega^2(A)=\omega_2(A)+\omega_1(A)\leq0$. However, $L^2(A(\cdot +it))$ is an even convex function with respect to $t$, implying $\omega^2(A)\geq 0$. Hence, $\omega^2(A)=0$, and $A \in \mathcal{R}(\Sp(4,\R))$. Theorem \ref{WEAK SIMPLE DENSITY} can then be applied to complete the proof.

\smallskip \textbf{Case III: $\omega^2(A)=0$ and $L_1(A)>L_2(A)=0$.}

Since $L_2(A)=0$ necessitates $\omega_2(A)\geq 0$, we have $\omega^1(A)=\omega^2(A)-\omega_2(A) \leq 0$. As $A\in C^\omega(\T,\Sp(4,\R))$, it follows that $\omega^1(A) \geq 0$. Consequently, $\omega^1(A)=0$, and $A \in \mathcal{R}(\Sp(4,\R))$. Theorem \ref{WEAK SIMPLE DENSITY} can be applied again to finish the proof.

\smallskip \textbf{Case IV: $\omega^2(A)=0$ and $L_1(A)=L_2(A)>0$.}

By Theorem \ref{DOMINATE=REGULAR}, $A \in \mathcal{UH}(\Sp(4,\R))$. Proposition \ref{DIAGNALIZATION LEMMA} ensures the existence of $T \in C^\omega( \chi \T,\Sp(4,\R))$, $\lambda \in C^\omega(\T,\R_+)$, and $\Lambda \in C^\omega(\chi \T,\SL(2,\R))$, where $\chi=1$ or $2$, such that $$T(x+\alpha)^{-1}A(x)T(x)=\begin{pmatrix}
\lambda(x)\Lambda(x) & O \\
O & \lambda(x)^{-1}\Lambda(x)^{-\top}
\end{pmatrix}.$$

If $\tau(E^u(A))=\tau(E^s(A))=1$, then $\chi=1$. Theorem \ref{JMD POSITIVE DENSITY} and Theorem \ref{CONTIOUS LYAPUNOV} allow us to choose $\tilde{\Lambda}$ close to $\Lambda$ with $0$. Define $$\tilde{A}(x)=T(x+\alpha)\operatorname{diag}(\lambda(x)\tilde{\Lambda}(x),\lambda(x)^{-1}\tilde{\Lambda}(x)^{-\top})T(x)^{-1}.$$
Then $\tilde{A} \in C^\omega(\T, \Sp(4,\R))$ is close to $A$ and satisfies $L_1(\tilde{A})>L_2(\tilde{A})>0$.

If $\tau(E^u(A))=\tau(E^s(A))=-1$, then $\chi=2$, and we have $T(x+1)=T(x)\mathcal{P}$ and $\Lambda \in \mathfrak{A}$. If $\omega(\Lambda)=0$, Proposition \ref{ANTIPERIODIC POSITIVE DENSE} enables us to select $\tilde{\Lambda}$ close to $\Lambda$ with $L(\tilde{\Lambda})>0$. Otherwise, if $\omega(\Lambda)>0$, Proposition \ref{RARECRITICAL} (equivalently, the fact that $\{\Lambda \in \mathfrak{A}: \omega(\Lambda)>0,L(\Lambda)=0\}$ has no interior in $ \mathfrak{A}$) ensures the existence of $\tilde{\Lambda}$ close to $\Lambda$ with $L(\tilde{\Lambda})>0$. Thus, $\tilde{A}$ as defined above is in $C^\omega(\T, \Sp(4,\R))$, close to $A$, and satisfies $L_1(\tilde{A})>L_2(\tilde{A})>0$.

  Finally, observe that the ensemble of $1$-regular or $2$-regular $\Sp(4,\mathbb{R})$-cocycles constitutes an open set within the domain of all analytic $\Sp(4,\mathbb{R})$-cocycles. Consequently, in the aforementioned four scenarios, we are able to select $\tilde{A}$ sufficiently close to $A$ such that $\omega^1(\tilde{A})\omega^2(\tilde{A})=0$. Thus, we have successfully concluded the proof.
\end{proof}

\subsection{Proof of Proposition \ref{ANTIPERIODIC POSITIVE DENSE-1}}


 Lemma \ref{DEGREE=0} implies that  $A \in \mathscr{A}$ is homotopic to the identity, then its fibered rotation number is well-defined. The key observation here is the following:
\begin{lemma}\label{zrho}
    For any $A \in \mathscr{A}$, we have $2\rho(\alpha,A)=0 \mod 1$. 
\end{lemma}
\begin{proof}
Denote the lift of $(\alpha, A)$ by
$\widetilde{F} : \mathbb T \times \mathbb R \rightarrow \mathbb T \times \mathbb R $, which is
of the form
$\tilde{F}(x,t) = (x + \alpha,t + f(x,t))$, such that 
\begin{equation}\label{ACTION}
    A(x)\binom{\cos (2 \pi y)}{\sin (2 \pi y)}=\lambda_A(x, y)\binom{\cos (2 \pi(y+f(x, y)))}{\sin (2 \pi(y+f(x, y)))}.
\end{equation}
Denote $\psi_N(x, t)=\sum_{n=0}^{N-1} f\left(\tilde{F}^n(x, t)\right)$, which satisfy 
\begin{equation}\label{pn}
    \psi_N(x,t+1)=\psi_N(x,t).
\end{equation}

Take $y=\frac{1}{2}-t$ in  \eqref{ACTION}, we have the following:
\begin{equation}\label{A(x+Alpha)}
    A(x+\frac{1}{2})\binom{\cos (2 \pi (\frac{1}{2}-t))}{\sin (2 \pi (\frac{1}{2}-t))}=\lambda_A(x+\frac{1}{2}, \frac{1}{2}-t)\binom{\cos (2 \pi(\frac{1}{2}-t+f(x+\frac{1}{2}, \frac{1}{2}-t)))}{\sin (2 \pi(\frac{1}{2}-t+f(x+\frac{1}{2}, \frac{1}{2}-t)))},
\end{equation}
Meanwhile, as $A \in \mathscr{A},$ by  \eqref{ACTION} we also have
\begin{equation}\label{A(x+Alpha)1}
    A(x+\frac{1}{2})\binom{\cos \left(2 \pi\left(\frac{1}{2}-t\right)\right)}{\sin \left(2 \pi\left(\frac{1}{2}-t\right)\right)}=\lambda_A(x, t)\binom{\cos \left(2 \pi\left(\frac{1}{2}-t-f(x, t)\right)\right.}{\sin \left(2 \pi\left(\frac{1}{2}-t-f(x, t)\right)\right.}.
\end{equation}
Thus \eqref{A(x+Alpha)} and \eqref{A(x+Alpha)1} imply that there exists $ k_{x t} \in \mathbb{Z}$ such that
$$
\frac{1}{2}-t-f(x, t)=\frac{1}{2}-t+f\left(x+\frac{1}{2}, \frac{1}{2}-t\right)+k _{x t}.
$$
By continuity, $k_{x t} \equiv K \in \Z$, implies that for any $x \in \T, t \in \R$, $$f(x, t)+f\left(x+\frac{1}{2}, \frac{1}{2}-t\right)=K.$$

Next, we prove $\psi_N(x, t)+\psi_N\left(x+\frac{1}{2}, \frac{1}{2}-t\right)=N K$  by  induction. Indeed, by \eqref{pn}, we have
$$
\begin{aligned}
& \psi_{N+1}(x, t)+\psi_{N+1}(x+\frac{1}{2}, \frac{1}{2}-t) \\
= & \psi_N(x+\alpha, t+f(x, t))+f(x, t)+\psi_N\left(x+\frac{1}{2}+\alpha, \frac{1}{2}-t+f(x+\frac{1}{2}, \frac{1}{2}-t)\right)+f(x+\frac{1}{2}, \frac{1}{2}-t) \\
= & \psi_N(x+\alpha, t+f(x, t))+\psi_N(x+\frac{1}{2}+\alpha, \frac{1}{2}-(t+f(x, t))+K)+K \\
= & \psi_N(x+\alpha, t+f(x, t))+\psi_N(x+\frac{1}{2}+\alpha, \frac{1}{2}-(t+f(x, t)))+K  \\
= & (N+1) K.
\end{aligned}
$$

This immediately implies that 
$$2\rho(\alpha,A)=\lim _{N \rightarrow +\infty} \frac{\psi_N(x, t)+\psi_N\left(x+\frac{1}{2}, \frac{1}{2}-t\right)}{N}=K\equiv0 \quad(\operatorname{mod} 1),$$
the result follows.
\end{proof}

\subsubsection{Case 1: $\beta(\alpha)=0$}

In the case $\beta(\alpha)=0$, by ART,  we have the following:

\begin{theorem}\cite{Aab,avila2}\label{DIOPHATINEARC}
     Let $\alpha \in \R\backslash\Q$ with $\beta(\alpha)=0$,  $A \in C^\omega(\T,\SL(2,\R))$, $k_0 \in \mathbb{Z}$. Suppose that $(\alpha,A)$ is subcritical
and $\mathbf{l}=2\rho(\alpha,A)-k_0 \alpha \in \mathbb{Z}.$  Then there exist $B \in C^\omega(\T,\PSL(2,\R)), d \in \R$ such that $\operatorname{deg} B=k_0$ and
$$(-1)^{\mathbf{l}}B(x+\alpha)^{-1}A(x)B(x)=I_2+d\mathcal{L}:=I_2+d\begin{pmatrix}
    0 & 1 \\
    0 & 0
\end{pmatrix} .$$
\end{theorem}

The following elementary observation is useful for us:

\begin{lemma}\label{degree}
    For any $B \in C(\T,\PSL(2,\R))$, we have the following:
    \begin{itemize}
    \item $B \in C(\T,SL(2,\R))$ if and only if $\deg B$ is even.
    \item $B \in C(2\T,SL(2,\R))- C(\T,SL(2,\R))$ if and only if $\deg B$ is odd.
    \end{itemize}
\end{lemma}
\begin{remark}
    This result was first proved in \cite{Ge-Jito-You}, we give another short proof here. 
\end{remark}
\begin{proof}
For any $B \in C(2\T,SL(2,\R))$, $B(x)$ is homotopic to the rotation matrix $R_{\frac{d}{2}x}$, with $d:=\deg B$ representing the topological degree of the map $B(\cdot)$. Specifically, if $B(x+1) = B(x)$, it implies that $R_{\frac{d}{2}(x+1)} = R_{\frac{d}{2}x}$, which necessitates that $d$ is an even number. Meanwhile, if $B(x+1) = -B(x)$, then $R_{\frac{d}{2}(x+1)} = -R_{\frac{d}{2}x}$, indicating that $d$ is odd. Given that $\PSL(2,\R) = \SL(2,\R)/\{\pm \id\}$, the conclusion that the degree $d$ is either even or odd, depending on the condition $B(x+1) = \pm B(x)$, is thus established.
\end{proof}

\begin{proof}
    By Theorem \ref{DIOPHATINEARC},  we have 
    \begin{equation}\label{re1}
    B(x+\alpha)^{-1}A(x)B(x)=I_2+d\mathcal{L}.\end{equation}
To prove Proposition \ref{ANTIPERIODIC POSITIVE DENSE}, the difficulty lies in the conjugacy $B(x)$ not necessary belongs to $\mathscr{A}$, thus one can't easily perturbate $A(x)$ to $\tilde{A}(\cdot) \in \mathscr{A}$ having positive Lyapunov exponent. Note by Lemma \ref{zrho}, $2\rho(\alpha,A)=0 \mod 1$, which gives $k_0=0$, then $\deg B=0$, by Lemma \ref{degree}, we have $B \in C^{\omega}(\T,SL(2,\R))$.
The key observation here is that even if  $B(x)$ doesn't  belong to $\mathscr{A}$, there exists 
constant $ T \in \SL(2,\R)$ such that $B(x)T \in \mathscr{A}.$ To prove this, we set 
\begin{equation}\label{DEF OF S} 
    S(x)=B(x)^{-1}\mathcal{T}B(x+\frac{1}{2}),
\end{equation}
     then $S \in C^\omega(\T,\GL(2,\R))$ with $\operatorname{det}S(x)=-1$ for any $x \in \T$. Consequently,
     combining  \eqref{DEF OF S} the fact that $A \in \mathscr{A}$,
     we have 
     \begin{eqnarray}\label{xal1}
     \mathcal{T} B\left(x+\frac{1}{2}+\alpha\right)&=&B(x+\alpha) S(x+\alpha),\\
    \label{xal2} A(x)B(x)S(x)=A(x)\mathcal{T}B(x+\frac{1}{2})&=& \mathcal{T}A(x+\frac{1}{2}) B(x+\frac{1}{2}) 
\end{eqnarray}
We then distinguish the proof into two cases: 

\smallskip
\textbf{Case I: $d=0.$} In this scenario, by \eqref{re1} and \eqref{xal2}, we have 
$$
\mathcal{T} B\left(x+\frac{1}{2}+\alpha\right)=B(x+\alpha) S(x).$$ 
Combine this with \eqref{xal1}, one gets the conclusion that $S(x+\alpha)=S(x)$, a fact that, combined with ergodicity, reveals the existence of a matrix $S \in \GL(2,\R)$ for which $S(x)\equiv 
S$ across all $x \in \T$. Moreover, by \eqref{DEF OF S}, we have $$
\mathcal{T}B(x+\frac{1}{2})=B(x) S(x)=\mathcal{T}^{-1} B(x-\frac{1}{2}) S(x-\frac{1}{2}) S(x)= \mathcal{T} B(x+\frac{1}{2}) S(x-\frac{1}{2}) S(x)
,$$
which implies that 
$S(x-\frac{1}{2}) S(x)=I_2$, this further implies that $S^2=I_2$, and because $\operatorname{det}S=-1$, there exists a matrix $ T \in \SL(2,\R)$ such that $T^{-1}ST=\mathcal{T}$.

Thus if we take $\Gamma(x)=B(x)T$, then by \eqref{DEF OF S},  $\Gamma(x+\frac{1}{2})=\mathcal{T}^{-1}\Gamma(x)\mathcal{T}$
which just means $\Gamma(x) \in \mathscr{A}.$ Meanwhile, by \eqref{re1}, we have 
 $\Gamma(x+\alpha)^{-1}A(x)\Gamma(x)=I_2,$  thus 
we define a sequence of matrices $\{A^{(n)}\}$ as follows: $$A^{(n)}(x)=\Gamma(x+\alpha)\operatorname{diag}(e^{\frac{1}{n}}, e^{-\frac{1}{n}})\Gamma(x)^{-1},$$ which converges to $A$ in the $C^\omega$-topology. Each $A^{(n)}$ satisfies $L(A^{(n)})>0$ and, importantly, $A^{(n)} \in \mathscr{A}$ as $\Gamma(x) \in \mathscr{A}.$ By upper-semicontinuity of the acceleration, $\omega(A^{(n)})=0$ for sufficiently large $n$.

\smallskip
\textbf{Case II: $d\neq 0$. } Similar argument as the first case, by \eqref{re1} and \eqref{xal2}, we have  $$B(x+\alpha)(I_2+d\mathcal{L}) S(x)=\mathcal{T} B(x+\frac{1}{2}+\alpha)(I_2+d\mathcal{L}).$$ Combining this with equation \eqref{xal1} leads to the conclusion that
$$S(x+\alpha)(I_2+d\mathcal{L})=(I_2+d\mathcal{L})S(x).$$
Thus if we denote $S(x)=\begin{pmatrix}
    s_1(x) & s_2(x) \\
    s_3(x) & s_4(x) \\
\end{pmatrix}$, we get
\begin{equation}\label{S1}
    s_1(x+\alpha)-s_1(x)=ds_3(x), 
\end{equation}
\begin{equation}\label{S2}
    s_2(x+\alpha)-s_2(x)=d s_4(x)-d s_1(x+\alpha),
\end{equation}
\begin{equation}\label{S3}
   s_3(x+\alpha)-s_3(x)=0,
\end{equation}
\begin{equation}\label{S4}
   s_4(x+\alpha)-s_4(x)=-d s_3(x+\alpha).
\end{equation}
By \eqref{S3} and ergodicity, it follows that $s_3(x)$ is a constant. By \eqref{S1}, we deduce that $s_3(x)\equiv 0$. Similarly, by ergodicity,  \eqref{S1} and  \eqref{S4} imply there exist two constants $s_1,s_4 \in \R$ such that $s_1(x)\equiv s_1$  and $s_4(x)\equiv s_4$.  Upon integrating both sides of \eqref{S2}, we find that $s_1=s_4$. Consequently, $\operatorname{det}S(x)=s_1s_4>0$, a contradiction.

\end{proof}

\subsubsection{Case 2: $\beta(\alpha)>0$}
In the case $\beta(\alpha)>0$, then qualitative result as Theorem \ref{DIOPHATINEARC} is not sufficient for us, then quantitative estimates come into play.

Suppose that $(\alpha,A)$ is subcritical in the strip $|\Im x|< h_*$,
there exists $\delta^{\prime}>0$, a subsequence $\left\{q_n\right\}$ consisting of the denominators of the continued fraction expansion of $\alpha$, with $q_{n+1}>e^{(\beta(\alpha)-o(1)) q_n}$, and 
\begin{equation}\label{EPSILON,DELTA,HN}
    \varepsilon_n=2 e^{\frac{2 \pi h *}{1+\alpha}} e^{-\frac{1}{2} \delta^{\prime} q_n}, \quad \delta_n=\frac{1}{2^n} h_*, \quad h=\frac{1}{2(1+\alpha)} h_*,
\end{equation}
such that the following result holds:
\begin{theorem}\cite{ARC}\label{LIOVILLEANARC}
 Let $\alpha \in \R\backslash\Q$ with $\beta(\alpha)>0$, $k_0 \in \mathbb{Z}$  and suppose that $(\alpha,A)$ is subcritical
and $\mathbf{l}=2\rho(\alpha,A)-k_0 \alpha \in \mathbb{Z}.$
Then there exists  $d_n \in \mathbb{R}$,  $B_n \in C_{h}^\omega(\T, \PSL(2, \R))$ with $\operatorname{deg} B_n=k_0$  such that
$$(-1)^{\mathbf{l}}B_n(x+\alpha)^{-1} A(x) B_n(x)=I_2+d_n\mathcal{L}+F_n(x),$$
with estimates
\begin{equation}\label{ESTIMATES OF DBF}
    \left|d_n\right| \leqslant e^{-q_{n+1} \varepsilon_n^{\frac{1}{4}}}, \left\|B_n\right\|_{h} \leqslant e^{2 q_{n+1}\varepsilon_n^{\frac{1}{4}}},\left\|F_n\right\|_{h} \leqslant e^{-q_{n+1} \delta_n}.
\end{equation}
\end{theorem}

The following result is  elementary but crucial for us:
\begin{lemma}\label{KYLINZHOU CRUCIAL LEMMA}
Let $q_n$ to be a subsequence of the  denominators of the continued fraction approximations  of $\alpha$, with the property
$q_{n+1}>e^{(\beta(\alpha)-o(1))q_n}$. 
Let ${\phi_n}$ to be a sequence of functions in $C_h^\omega(\T, \C)$ satisfying for some $c>0$,
\begin{eqnarray*}
\left\|\phi_n(x)\right\|_h &\lesssim& e^{4 q_{n+1} \varepsilon_n^{\frac{1}{4}}},\\
\left\|\phi_n(x+\alpha)-\phi_n(x)\right\|_h &\lesssim& e^{-2cq_{n+1}\delta_n},
\end{eqnarray*}
then for any $0<h'<h$, we have
$$\left\|\phi_n(x)-\int_0^1\phi_{n}(x)dx\right\|_{h'} \lesssim \frac{1}{h-h'}e^{-cq_{n+1}\delta_n}.$$
\end{lemma}
\begin{proof}

Let $\varphi_n(x)=\phi_{n}(x+\alpha)-\phi_{n}(x)$. 
The assumption yields $$\left|\hat{\varphi}_n(k)\right|\leq \left\|\varphi_n\right\|_{h}
e^{-2\pi |kh|}\lesssim e^{-2cq_{n+1}\delta_n-2\pi |kh|},$$ $$\left|\hat{\phi}_n(k)\right|\leq \left\|\phi_n\right\|_{h}e^{-2\pi |kh|}\lesssim e^{4q_{n+1}\varepsilon_n^{\frac{1}{4}}-2\pi |kh|}.$$ For $0<|k|<q_{n+1}$, since $\|k\alpha\|_{\R/\Z}\geq\frac{1}{2q_{n+1}}$, from the fact that $\hat{\varphi}_n(k)=\hat{\phi}_n(k)(e^{2k\pi i\alpha}-1)$, we deduce the estimation $$\left|\hat{\phi}_n(k)\right|\leq\frac{\left|\hat{\varphi}_n(k)\right|}{\left| e^{2k\pi i\alpha}-1 \right|}\leq\frac{\left|\hat{\varphi}_n(k)\right|}{\|k\alpha\|_{\R/\Z}}\lesssim q_{n+1} e^{-2cq_{n+1}\delta_n-2\pi |kh|}.$$ Therefore, we can estimate
$$
\begin{aligned}
    \left\|\phi_{n}(x)-\hat{\phi}_n(0)\right\|_{h'}
    &\leq \sum_{0<|k|<q_{n+1}}\left|\hat{\phi}_n(k)\right| e^{2\pi |kh^\prime|}+\sum_{|k| \geq q_{n+1}}\left|\hat{\phi}_n(k)\right|e^{2\pi |kh^\prime|}\\
    &\lesssim \sum_{0<|k|< q_{n+1}}q_{n+1}e^{-2cq_{n+1}\delta_n-2\pi |k|(h-h^\prime)}+\sum_{|k| \geqslant q_{n+1}} e^{{4q_{n+1}\varepsilon_n^{\frac{1}{4}}}-2\pi |k|(h-h^\prime)}\\
    &\lesssim  \frac{1}{1-e^{-2 \pi (h-h^{\prime})}}(q_{n+1}e^{-2cq_{n+1}\delta_n}+e^{{4q_{n+1}\varepsilon_n^{\frac{1}{4}}}-2\pi q_{n+1}(h-h^\prime)}).\\
    &\lesssim  \frac{1}{h-h'}e^{-cq_{n+1}\delta_n}.
\end{aligned}
$$
 The last inequality holds since by  \eqref{EPSILON,DELTA,HN} and $q_{n+1}>e^{(\beta(\alpha)-o(1))q_n}$, we have $q_{n+1}\lesssim e^{cq_{n+1}\delta_n}$ and $e^{{4q_{n+1}\varepsilon_n^{\frac{1}{4}}}-2\pi q_{n+1}(h-h^\prime)}\lesssim e^{-cq_{n+1}\delta_n}$.
\end{proof}

Now we finish the proof of Proposition \ref{ANTIPERIODIC POSITIVE DENSE} for the case $\beta(\alpha)>0$.

\begin{proof}
By Theorem \ref{LIOVILLEANARC}, we have 
\begin{equation}\label{re2}
    B_n(x+\alpha)^{-1} A(x) B_n(x)=I_2+d_n\mathcal{L}+F_n(x).
\end{equation}
We will follow the similar scheme as the case $\beta(\alpha)=0$, and
we set 
\begin{equation}\label{sndef}
S_n(x)=B_n(x)^{-1}\mathcal{T} B_n(x+\frac{1}{2})
\end{equation}
first by the same argument as above,  $B_n \in C_{h}^\omega(\T, \SL(2, \R))$, which imply  $S_n \in C^\omega
(\T,\GL(2,\R))$ with $\operatorname{det}S_n(x)=-1$ for any $ x \in \T$. By \eqref{ESTIMATES OF DBF}, we have an estimation 
\begin{equation}\label{SIZE OF S}
    \left\|S_n(x)\right\|_h \lesssim e^{4 q_{n+1} \varepsilon_n^{\frac{1}{4}}}.
\end{equation} By the fact $A \in \mathscr{A}$ we have the following equations similar to \eqref{xal1}, \eqref{xal2}:
\begin{eqnarray}\label{xal3}
     \mathcal{T} B_n\left(x+\frac{1}{2}+\alpha\right)&=&B_n(x+\alpha) S_n(x+\alpha),\\
    \label{xal4} A(x)B_n(x)S_n(x)&=& \mathcal{T}A(x+\frac{1}{2}) B_n(x+\frac{1}{2}). 
\end{eqnarray}
We then distinguish the proof into two cases:\\

\smallskip
\textbf{Case I:}
$\left|d_n\right| \leqslant e^{-\frac{1}{6} q_{n+1} \delta_n}$ for any $n \in \mathbb{N}_{+}.$

In this scenario we set $Y_n(x)=F_n(x)+d_n\mathcal{L}$, then we have the estimation $\left\|Y_n(x)\right\|_h \lesssim e^{-\frac{1}{6} q_{n+1} \delta_n}.$
By \eqref{re2} and \eqref{xal4}, we have 
$$\mathcal{T}  B_n(x+\frac{1}{2}+\alpha)(I_2+Y_n(x+\frac{1}{2}))=B_n(x+\alpha)(I_2+Y_n(x)) S_n(x)$$
Combining this with \eqref{xal3}, one gets the conclusion that $S_n(x+\alpha)(I_2+Y_n(x+\frac{1}{2}))=(I_2+Y_n(x)) S_n(x)$, implies that
\begin{equation}\label{ESTIMATE}
    \left\|S_n(x+\alpha)-S_n(x)\right\|_h=\left\|Y_n(x) S_n(x)-S_n(x+\alpha) Y_n(x+\frac{1}{2})\right\|_h \lesssim e^{-\frac{1}{8} q_{n+1} \delta_n}.
\end{equation}
By \eqref{SIZE OF S} and \eqref{ESTIMATE},  Lemma \ref{KYLINZHOU CRUCIAL LEMMA} gives us 
\begin{equation}\label{NEW SIZE OF SN}
    \left\|S_n(x)-\hat{S}_n(0)\right\|_{h/2} \lesssim e^{-\frac{1}{16} q_{n+1} \delta_n}.
\end{equation}
Moreover, by \eqref{sndef}, we have $$
\mathcal{T}  B_n(x+\frac{1}{2})=B_n(x) S_n(x)=\mathcal{T}B_n(x-\frac{1}{2}) S_n(x-\frac{1}{2}) S_n(x)
$$ which implies that
\begin{equation}\label{SN2=I2}
    S_n(x-\frac{1}{2}) S_n(x)=I_2.
\end{equation}
Let $ \tilde{S}_n(x)=S_n(x)-\hat{S}_n(0)$, by  \eqref{NEW SIZE OF SN}, \eqref{SN2=I2} we have an estimation 
\begin{equation*}
    \left\|\hat{S}_n(0)^2-I_2\right\|=\left\|\hat{S}_n(0) \tilde{S_n}(x)+\tilde{S_n}(x+\frac{1}{2}) \hat{S}_n(0)+\tilde{S_n}(x+\frac{1}{2}) \tilde{S_n}(x)\right\|_{h/2} \lesssim e^{-\frac{1}{18} q_{n+1} \delta_n}.
\end{equation*} 
Since $\operatorname{det}(S_n(x))=-1$, by \eqref{NEW SIZE OF SN}
we know that $\operatorname{det} \hat{S}_n(0)<0$ whenever $n$ is large enough, hence $\hat{S}_n(0)$ has two distinct eigenvalues $\lambda_n>0>k_n,$ with estimate 
\begin{equation}\label{LAMDA K EST}
    \max \left\{\left|k_n+1\right|,\left|\lambda_n-1\right|\right\} \lesssim e^{-\frac{1}{18} q_{n+1} \delta_n}.
\end{equation}
Meanwhile, there exists $ T_n \in \SL(2, \mathbb{R})$, such that
     \begin{equation}\label{T-1ST}
         T_n^{-1} \hat{S}_n(0) T_n=\left(\begin{array}{cc}k_n & 0 \\ 0 & \lambda_n\end{array}\right).
     \end{equation}
by \cite[Lemma A.6]{Eli02}, one has estimate
 \begin{equation}\label{SIZE OF T}
     \left\|T_n\right\| \lesssim \max \left\{1, \frac{\left\|\hat{S}_n(0)\right\|}{\left|k_n-\lambda_n\right|}\right\} \lesssim e^{4 q_{n+1} \varepsilon_n^{\frac{1}{4}}},
 \end{equation}

Define $M_n \in C^\omega(\T,\SL(2,\R))$ by $M_n(x)=B_n(x)T_n$, then by 
\eqref{sndef}, we have 
\begin{eqnarray}
\nonumber \mathcal{T} M_n(x+\frac{1}{2}) \mathcal{T}^{-1}
&=&\mathcal{T} B_n(x+\frac{1}{2}) T_n  \mathcal{T}^{-1} 
=M_n(x)\left(T_n^{-1} \hat{S}_n(0) T_n+T_n^{-1} \tilde{S_n}(x) T_n\right) \\
\label{TBT} &=& M_n(x) +G_n(x),
\end{eqnarray}
where $G_n(x)=(M_n(x)\operatorname{diag}(k_n+1,\lambda_n-1)+M_n(x)T_n^{-1} \tilde{S_n}(x)T_n)\mathcal{T}.$ 
Now, by \eqref{TBT} and the fact that ${M_n}(x+1)=M_n(x)$, we have
\begin{eqnarray*}
    \mathcal{T}G_n(x+\frac{1}{2})\mathcal{T}^{-1}& =&-G_n(x),\\
    \mathcal{T}\left(M_n(x+\frac{1}{2})+\frac{1}{2} G_n(x+\frac{1}{2})\right)\mathcal{T}^{-1} &=& M_n(x)+\frac{1}{2} G_n(x).
\end{eqnarray*}
Therefore,  we set $\Gamma_n(x)=M_n(x)+\frac{1}{2} G_n(x)$, then 
$\Gamma_n(x+\frac{1}{2})=\mathcal{T}^{-1}   \Gamma_n(x)\mathcal{T} .$ 
This means $\Gamma_n(x)$ already satisfy the group structure, the only problem is that its determinant is not $1$.

To deal with this issue, quantitative estimates are involved. First    by \eqref{ESTIMATES OF DBF}, \eqref{NEW SIZE OF SN}, \eqref{LAMDA K EST} and \eqref{SIZE OF T}, one can get the estimations
\begin{equation}\label{EST OF B G}
    \|M_n(x)\|_{h/2} \lesssim e^{6q_{n+1}\varepsilon
_n^{\frac{1}{4}}}, \quad \|G_n(x)\|_{h/2}\lesssim e^{-\frac{1}{20}q_{n+1}\delta_n}.
\end{equation}
As $M_n \in C^\omega(\T,\SL(2,\R))$, this imply that 
\begin{equation}\label{EST GAMMA}
    \left\|\Gamma_n(x)\right\|_{h/2} \lesssim e^{8 q_{n+1} \varepsilon_n^{ \frac{1}{4}}}, \quad \left\|\operatorname{det}\Gamma_n(x)-1\right\|_{h/2} \lesssim e^{-\frac{1}{22}q_{n+1}\delta_n}, \quad \left\|\Gamma_n(x)^{-1}\right\|_{h/2} \lesssim e^{8 q_{n+1} \varepsilon_n^{ \frac{1}{4}}}.
\end{equation}
Thus $\Gamma_n \in C^\omega(\T, \GL(2,\R))$ for sufficiently large $n$.
After that, under the action of $A(x)$:
$$
\begin{aligned}
A(x) \Gamma_n(x)&=A(x) B_n(x) T_n+\frac{1}{2} A(x) G_n(x)\\
& =B_n(x+\alpha)(I_2+Y_n(x)) T_n+\frac{1}{2} A(x) G_n(x) \\
& =\Gamma_n(x+\alpha)-\frac{1}{2} G_n(x+\alpha)+\frac{1}{2} A(x) G_n(x)+B_n(x+\alpha) Y_n(x) T_n.
\end{aligned}
$$
By \eqref{ESTIMATES OF DBF}, \eqref{EST OF B G} and \eqref{EST GAMMA}, we get the estimation 
\begin{equation}\label{NEW ARC}
    \left\|A(x)-\Gamma_n(x+\alpha) \Gamma_n(x)^{-1}\right\|_{h/2} \lesssim e^{-\frac{1}{22} q_{n+1} \delta_n}
\end{equation}

Finally, we take a sequence $\{A^{(n)}\}$ defined as follows:
$$A^{(n)}(x) = \Gamma_n(x+\alpha) \begin{pmatrix}
(1+e^{-q_{n+1} \delta_n})\frac{\operatorname{det}\Gamma_n(x)}{\operatorname{det}\Gamma_n(x+\alpha)} & 0 \\
0 & \frac{1}{1+e^{-q_{n+1} \delta_n}}
\end{pmatrix} \Gamma_n(x)^{-1},$$
It is straightforward to verify that $A^{(n)}(x+\frac{1}{2}) = \mathcal{T}^{-1}A^{(n)}(x)\mathcal{T}$, indicating that $A^{(n)}$ belongs to the set $\mathscr{A}$. Furthermore, the Lyapunov exponent of $A^{(n)}$ can be estimated as $L(A^{(n)}) = \ln(1+e^{-q_{n+1} \delta_n}) > 0$.

Using \eqref{EST GAMMA} and \eqref{NEW ARC}, we obtain the estimate:
$$\left\|A^{(n)}(x) - A(x) \right\|_{h/2} \lesssim e^{16q_{n+1}\varepsilon_n^{\frac{1}{4}}}(e^{-\frac{1}{22}q_{n+1}\delta_n}+e^{-\frac{1}{22}q_{n+1}\delta_n}) \lesssim e^{-\frac{1}{24}q_{n+1}\delta_n},$$
which implies that $A^{(n)}$ converges to $A$ in the $C_{h/2}^\omega$-topology. Employing the upper semicontinuity of the acceleration, we find that $\omega(A^{(n)}) = 0$ for sufficiently large $n$. Consequently, we are able to perturb $(\alpha,A)$ to a sequence of cocycles with the desired properties.

\smallskip
\textbf{Case II}:  there exists $\left\{n_k\right\} \subseteq N_{+}$ such that
\begin{equation}\label{SLOWLY DEACREASING D}
    \left|d_{n_k}\right| \geqslant e^{-\frac{1}{6} q_{n_k+1} \delta n_k}.
\end{equation}
By \eqref{re2} and \eqref{xal4}, we get
\begin{eqnarray*}
 \mathcal{T}  B_n(x+\frac{1}{2}+\alpha)\left(\left(I_2+d_n\mathcal{L}\right)+F_n(x+\frac{1}{2})\right)=B_n(x+\alpha)\left(\left(I_2+d_n\mathcal{L}\right)+F_n(x)\right) S_n(x).
\end{eqnarray*}
Combining this with \eqref{xal3}, the following equation holds:
\begin{equation}\label{SNPARABOLIC SMALL}
    S_n(x+\alpha)\left(I_2+d_n\mathcal{L}\right)-\left(I_2+d_n\mathcal{L}\right) S_n(x)=F_n(x) S_n(x)-S_n(x+\alpha) F_n(x+\frac{1}{2})
\end{equation}
We write write $S_n(x)$ as
\begin{equation}\label{We write Sn(x) as}
    S_n(x)=\left(\begin{array}{ll}s_1(x, n) & s_2(x, n) \\ s_3(x, n) & s_4(x, n)\end{array}\right),
\end{equation} then denote $\tilde{s}_j(x,n)=s_j(x,n)-\hat{s}_j(0,n),$ for $j=1,2,3,4,$ and let $$ \quad F_n(x) S_n(x)-S_n(x+\alpha) F_n(x+\frac{1}{2})=\left(\begin{array}{ll}f_1(x, n) & f_2(x, n) \\ f_3(x, n) & f_4(x, n)\end{array}\right),$$ then \eqref{SNPARABOLIC SMALL} can be rewritten as 
\begin{equation}\label{EQ S1}
    s_1(x+\alpha, n)-s_1(x, n)=d_n s_3(x, n)+f_1(x, n)
\end{equation}
\begin{equation}\label{EQ S2}
    s_2(x+\alpha,n)-s_2(x, n)=d_n s_4(x,n)-d_n s_1(x+\alpha,n)+f_2(x, n)
\end{equation}
\begin{equation}\label{EQ S3}
   s_3(x+\alpha,n)-s_3(x,n)=f_3(x,n)
\end{equation}
\begin{equation}\label{EQ S4}
   s_4(x+\alpha, n)-s_4(x,n)=-d_n s_3(x+\alpha,n)+f_4(x,n)
\end{equation}
By \eqref{ESTIMATES OF DBF},  $\left\|f_j(x, n)\right\|_h \lesssim e^{-\frac{5}{6} q_{n+1} \delta_n},(j=1,2,3,4).$ Integrating both sides of \eqref{EQ S1}, we get $d_n \hat{s}_3(0, n)+\hat{f}_1(0, n)=0$, thus
\begin{equation}\label{s3ESTIMATE}
    \left|\hat{s}_3(0, n_k)\right|=\left|\frac{\hat{f}_1(0, n_k)}{d_{n_k}}\right| \lesssim e^{-\frac{2}{3} q_{n_k+1} \delta_{n_k}}.
\end{equation}
By \eqref{SIZE OF S}, we can use Lemma \ref{KYLINZHOU CRUCIAL LEMMA} to demonstrate that  $ \left\|\tilde{s}_3(x, n)\right\|_{\frac{h}{2}} \lesssim e^{-\frac{5}{12} q_{n+1} \delta_n}$ from the equation \eqref{EQ S3}, thus by \eqref{s3ESTIMATE}, we have
\begin{equation}\label{SIZE OF S3(x)}
     \left\|s_3(x, n)\right\|_{\frac{h}{2}} \lesssim e^{-\frac{5}{12} q_{n+1} \delta_n}
\end{equation}
 Again by \eqref{SIZE OF S} and \eqref{SIZE OF S3(x)}, from \eqref{EQ S1} \eqref{EQ S4} we can use Lemma \ref{KYLINZHOU CRUCIAL LEMMA} to demonstrate that
 \begin{equation}\label{SIZE OF S1}
     \left\|\tilde{s}_1(x, n_k)\right\|_{\frac{h}{4}} \lesssim e^{-\frac{1}{3} q_{n_k+1} \delta_{n_k}}, \quad \left\|\tilde{s}_4(x, n_k)\right\|_{\frac{h}{4}} \lesssim e^{-\frac{1}{3} q_{n_k+1} \delta_{n_k}}.
\end{equation}
Integrating both sides of \eqref{EQ S2} gives that $0=d_n \hat{s}_4(0, n)-d_n \hat{s}_1(0, n)+\hat{f}_2(0, n),$ by \eqref{SLOWLY DEACREASING D} we get the estimation 
\begin{equation}\label{S1S4>0}
    \left|\hat{s}_4\left(0, n_k\right)-\hat{s}_1(0, n_k)\right|=\left|\frac{\hat{f_2}(0, n_k)}{d_{n_k}}\right| \lesssim e^{-\frac{2}{3} q_{n_k+1} \delta_{n_k}}.
\end{equation}
Finally, through computation, we have
$$
\begin{aligned}
-1&=\operatorname{det} S_n(x)\\ & =s_1(x, n) s_4(x, n)-{s}_2(x, n) s_3(x, n) \\
& =\hat{s}_1(0, n) \hat{s}_4(0, n)+\tilde{s}_1(x, n) \tilde{s}_4(x, n)+\tilde{s}_1(x, n) \hat{s}_4(0, n)+\tilde{s}_4(x, n) \hat{s}_1(0, n)-{s}_2(x, n) s_3(x, n),
\end{aligned}
$$
Since \eqref{SIZE OF S} gives that 
$\left\|s_2(x, n)\right\|_{\frac{h}{4}} \lesssim e^{4 q_{n+1} \varepsilon_n^{\frac{1}{4}}}, \left|\hat{s}_1(0, n)\right| \lesssim e^{4 q_{n+1} \varepsilon_n^{\frac{1}{4}}},
\left|\hat{s}_4(0, n)\right| \lesssim e^{4 q_{n+1} \varepsilon_n^{\frac{1}{4}}},$ by \eqref{SIZE OF S3(x)} and \eqref{SIZE OF S1} we can get an estimation $\left|\hat{s}_1(0, n) \hat{s}_4(0, n)+1\right| \lesssim e^{-\frac{1}{4} q_{n+1} \delta_n}$. Thus, $\hat{s}_1(0, n)$ $\hat{s}_4(0, n)<0$ whenever $n$ is large enough, implies that
$$
\begin{aligned}
\left|\hat{s}_1(0, n)-\hat{s}_4(0, n)\right|  \geqslant 2 \sqrt{-\hat{s}_1(0, n) \hat{s}_4(0, n)} 
 \geqslant 2 \sqrt{1-\left|1+\hat{s}_1(0, n) \hat{s}_4(0, n)\right|} 
 \end{aligned}
$$
which  converges to $2$, and contradicts with \eqref{S1S4>0}.

Combining the above two cases, we finish the whole proof.
\end{proof}

\section*{Acknowledgements} 
Qi Zhou is supported by National Key R\&D Program of China (2020 YFA0713300) and Nankai Zhide Foundation. Disheng Xu is supported by National Key R$\&$D Program of China No. 2024YFA1015100, NSFC 12090010
and 12090015.


\begin{thebibliography}{99}
\bibitem{AKLM} D. Alekseevsky, A. Kriegl, M. Losik, and P. W. Michor, Choosing roots of polynomials
smoothly, Israel J. Math. 105 (1998), 203–233.

\bibitem{Av2011}A. Avila, Density of positive Lyapunov exponents for $\SL(2,\mathbb R)$-cocycles. Journal of the American Mathematical Society, 24(4): 999-1014, 2011.

\bibitem{JMD} A. Avila. Density of positive Lyapunov exponents for quasiperiodic $\SL(2,\R)$-cocycles in arbitrary dimension. Journal of Modern Dynamics, 3(4): 631-636, 2009.

\bibitem{Aab}
A. Avila, Absolutely continuous spectrum for the almost Mathieu operator.  arXiv:0810.2965.

\bibitem {Aac}
A. Avila,
Almost reducibility and absolute continuity,
arXiv:1006.0704.

 \bibitem{Global}  A. Avila, Global theory of one-frequency Schrödinger operators. Acta Math, 215: 1-54, 2015

\bibitem{avila2} A. Avila, KAM, Lyapunov exponents,  and the spectral dichotomy for typical one-frequency Schr\"odinger operators. arXiv:2307.11071

\bibitem{AK}A. Avila, R. Krikorian: Monotonic cocycles. Inventiones mathematicae, 202(1): 271-331, 2015.

\bibitem{AJ}A. Avila, S. Jitomiskaya, The Ten Martini Problem, Annals of Mathematics
Second Series, 170(1): 303 -342, 2009.

\bibitem{1D} A. Avila, S. Jitomirskaya and C. Sadel. Complex one-frequency cocycles. Jounal of the European Mathematical Society, 16: 1915-1935, (2014)

\bibitem{ASV}A. Avila, J. Santamaria and M. Viana, Cocycles over partially hyperbolic maps. Astrisque, 358, 13-74.


\bibitem{AV}A. Avila and M. Viana: Simplicity of Lyapunov spectra: proof of the Zorich- Kontsevich conjecture. Acta mathematica, 198(1): 1-56, 2007.

\bibitem{ARC}A. Avila, J. You and Q. Zhou, Dry Ten Martini Problem in the non-critical case. arXiv:2306.16254.

\bibitem{APP1} M. Bauer, M. Bruveris, P. Harms and P. W. Michor, Smooth Perturbations of the Functional Calculus and Applications to Riemannian Geometry on Spaces of Metrics. Communications in Mathematical Physics, 389(2): 899–931, 2022.


\bibitem{BQ}Y. Benoist and J. F. Quint, 2016. Random walks on reductive groups (pp. 153-167). Springer International Publishing.

\bibitem{Bo}J. Bochi, Genericity of zero Lyapunov exponents. Ergodic Theorem and Dynamical Systems, 22: 1667–1696, 2002.

\bibitem{BoVi}J. Bochi and M. Viana, The Lyapunov Exponents of Generic Volume-Preserving and Symplectic Maps. Annals of Math, 161: 1423-1485, 2005.

\bibitem{BV}C. Bonatti and M. Viana: Lyapunov exponents with multiplicity 1 for deterministic products of matrices. Ergodic Theory and Dynamical Systems, 24(05): 1295-1330, 2004.

\bibitem{Cannas}A. Cannas da Silva, Lectures on Symplectic Geometry, Lecture Notes in Mathematics, Springer-Verlag, 2008.

\bibitem{APP2}P.T. Chruściel, E. Delay, P. Klinger, A. Kriegl, P.W. Michor and A. Rainer, Non-singular space-times with a negative cosmological constant: V. Boson stars. Letters in Mathematical Physics, 108(9): 2009–2030, 2018.

\bibitem{DF1}
D. Damanik, and J. Fillman, One-Dimensional Ergodic Schrödinger Operators: I. General Theory. Vol. 221. American Mathematical Society, 2022.

\bibitem{DF2}
D. Damanik, and J. Fillman, One-Dimensional Ergodic Schrödinger Operators: II. Specific Classes. Vol. 249. American Mathematical Society, 2024.

\bibitem{Eli02}
L. H. Eliasson, Perturbations of linear quasi-periodic system.
In {\it Dynamical Systems and Small Divisors} (pp. 1-60). Springer, Berlin, Heidelberg (2002).


\bibitem{F-K}
 B. Fayad  and R. Krikorian, 
Rigitidy results for quasiperiodic
$\SL(2,\R)$-cocycles,
Journal of Modern Dynamics, \textbf{3} no. 4 (2009),
479-510.

\bibitem{FK60}H. Furstenberg, H. Kesten: Products of random matrices. The Annals of Mathematical Statistics, 457-469, 1960. 


\bibitem{Fur63} H. Furstenberg, Noncommuting random products. Transactions of the American Mathematical Society, 377-428, 1963.

\bibitem{ge}L. Ge,  On the Almost Reducibility Conjecture. Geom. Funct. Anal. 34, 32–59 (2024). 

\bibitem{Ge-Jito-You}L. Ge, S. Jitomirskaya, J. You. Kotani theory, puig’s argument, and stability of the ten martini problem. arXiv:2308.09321

\bibitem{Ge-Jito}L. Ge, S. Jitomirskaya. Hidden subcriticality, symplectic structure, and universality of sharp arithmetic spectral results
for type I operators.  arXiv:2407.08866

\bibitem{GJYZ} L. Ge, S. Jitomirskaya, J. You and Q. Zhou. Multiplicative Jensens formula and quantitative global theory of one-frequency Schr\"odinger operators. arXiv:2306.16387 

\bibitem{GM}I. Y. Gol’dsheid and G. A. Margulis, Lyapunov indices of a product of random matrices. Russian mathematical surveys, 44(5): 11-71, 1989.

\bibitem{GR}Y. Guivarc’h, Y. and A. Raugi: Frontie\`re de Furstenberg, propri\'et\'es de contraction et th\'eor\`emes de convergence. Probability Theory and Related Fields, 69(2): 187-242, 1985.  


\bibitem{Harmer} M. Harmer. Hermitian symplectic geometry and extension theory. Journal of Physics A,
33 (2000): 9193–9203.

\bibitem{HP} A. Haro and J. Puig. A Thouless formula and Aubry duality for long-range Schr\"odinger skew-products. {\it Nonlineraity} 26(5): 1163-1187, 2013.

  

\bibitem{H} M. Herman. Une m\'ethode pour minorer les
exposants de Lyapounov et quelques exemples montrant le caract\`ere
local d'un th\'eor\`eme d'Arnol'd et de Moser sur le tore de dimension $2$. {\it Comment. Math. Helv.} {\bf 58(3)} (1983), 453-502.

\bibitem{Kato} T. Kato, Perturbation Theory for Linear Operators, Fisicalbook, Germany, 1995, pp.106.

\bibitem{KoSi}S. Kotani and B. Simon: Stochastic Schr\"odinger operators and Jacobi matrices on the strip. Communications in mathematical physics, 119(3):403-429, 1988.

\bibitem{Ko}S. Kotani: Ljaponov indices determine absolutely continuous spectra of stationary random one-dimensional Schr\"odinger operators. Stochastic analysis. Ito, K. (ed.) pp. 225-248. North Holland: Amsterdam, 1984.

\bibitem{KL} K. Kurdyka and L. Paunescu, Hyperbolic polynomials and multiparameter real-analytic per-
turbation theory, Duke Math. J. 141 (2008), no. 1, 123–149.

\bibitem{KMR}A. Kriegl, P. W. Michor, and A. Rainer,Denjoy–Carleman differentiable perturbation of polynomials and unbounded opera-
tors, Integral Equations and Operator Theory 71 (2011), no. 3, 407–416.

\bibitem{Kri}
 R. Krikorian, Reducibility, differentiable
rigidity and Lyapunov exponents for quasiperiodic cocycles on
$\mathbb{T}\times SL(2,\mathbb{R})$,
 preprint (www.arXiv.org,math.DS/0402333).

\bibitem{JICM}
S. Jitomirskaya. One-dimensional quasiperiodic operators: global theory,
  duality, and sharp analysis of small denominators.
\newblock In: I{CM}---{I}nternational {C}ongress of {M}athematicians. {V}ol. 2.
  {P}lenary lectures, pp. 1090--1120. EMS Press, Berlin (2023)

\bibitem{johonson and moser} R. Johnson and J. Moser. The rotation number for almost periodic potentials. Communications in mathematical physics, 84: 403–438, 1982.    


\bibitem{Long}Y. Long. Index Theory for Symplectic Paths with Applications. Vol. 207. Birkh\"auser, 2012.

\bibitem{Prasolov} V. Prasolov, Polynomials, pp:20-25. Springer, Germany, 2009. 

\bibitem{Rai1} A. Rainer, Perturbation of complex polynomials and normal operators, Math. Nach. 282
(2009), no. 12, 1623–1636.

\bibitem{Rai}A. Rainer, Perturbation theory for normal operators, Trans. Amer. Math. Soc., 365 (2013), 5545--5577.

\bibitem{Rai2}A. Rainer, Quasianalytic multiparameter perturbation of polynomials and normal matrices,
Trans. Amer. Math. Soc. 363 (2011), no. 9, 4945–4977.

\bibitem{APP3}N. Van Thu and H.-H. Vo, Stability and limiting properties of generalized principal eigenvalue for inhomogeneous nonlocal cooperative system. arXiv preprint, arXiv:2302.02861, 2023.		

\bibitem{V}M. Viana, Almost all cocycles over any hyperbolic system have nonvanishing Lyapunov exponents. Annals of Mathematics, 643-680, 2008.

\bibitem{ViaLec}M. Viana, 2014. Lectures on Lyapunov exponents (Vol. 145). Cambridge University Press.

\bibitem{DiSheng Xu}D. Xu. Density of positive Lyapunov exponents for symplectic cocycles. Journal of the European Mathematical Society, 21(10): 3143-3190, 2019.


\end{thebibliography}
\end{document}